\documentclass[11pt,a4paper]{amsart}
\pdfoutput=1
\usepackage[centering, margin=1in]{geometry}
\usepackage{accents,amssymb,enumitem,graphicx,mathrsfs,mathtools}
\usepackage{physics}
\usepackage[stretch=10]{microtype}
\usepackage{xcolor}
\definecolor{dark-red}{rgb}{0.4,0.15,0.15}
\definecolor{dark-blue}{rgb}{0.15,0.15,0.4}
\definecolor{medium-blue}{rgb}{0,0,0.5}
\usepackage[unicode]{hyperref}
\hypersetup{
	pdftitle={Subconvexity Implies Effective Quantum Unique Ergodicity for Hecke--Maa\texorpdfstring{\ss}{ß} Cusp Forms on \texorpdfstring{$\mathrm{SL}_2(\mathbb{Z}) \backslash \mathrm{SL}_2(\mathbb{R})$}{SL\texttwoinferior(Z)\textbackslash SL\texttwoinferior(R)}},
	pdfauthor={Ankit Bisain, Peter Humphries, Andrei Mandelshtam, Noah Walsh, and Xun Wang},
	pdfnewwindow=true,
	colorlinks, linkcolor={dark-red},
	citecolor={dark-blue}, urlcolor={medium-blue}
}
\allowdisplaybreaks
\newtheorem{theorem}{Theorem}[section]

\newtheorem{lemma}[theorem]{Lemma}
\newtheorem{corollary}[theorem]{Corollary}
\theoremstyle{remark}
\newtheorem{remark}[theorem]{Remark}
\newcommand{\GL}{\mathrm{GL}}
\newcommand{\SL}{\mathrm{SL}}
\DeclareMathOperator{\ad}{ad}
\DeclareMathOperator{\sgn}{sgn}
\newcommand*\pFq[4]{ {}_{#1}F_{#2}\left(\genfrac{}{}{0pt}{}{#3}{#4};1\right)}
\newcommand*\pFqx[5]{ {}_{#1}F_{#2}\left(\genfrac{}{}{0pt}{}{#3}{#4};#5\right)}
\newcommand{\gamer}[1]{\Gamma\left(#1\right)}
\newcommand{\half}{\frac{1}{2}}
\newcommand{\A}{\mathbb{A}}
\newcommand{\C}{\mathbb{C}}
\newcommand{\Hb}{\mathbb{H}}
\newcommand{\N}{\mathbb{N}}
\newcommand{\Q}{\mathbb{Q}}
\newcommand{\R}{\mathbb{R}}
\newcommand{\Z}{\mathbb{Z}}
\newcommand{\WW}{\mathcal{W}}
\newcommand{\domega}{\mathrm{d}\omega}
\newcommand{\dtheta}{\mathrm{d}\theta}
\newcommand{\dmu}{\mathrm{d}\mu}
\newcommand{\dx}{\mathrm{d}x}
\newcommand{\dy}{\mathrm{d}y}
\newcommand{\dt}{\mathrm{d}t}
\newcommand{\ds}{\mathrm{d}s}
\newcommand{\dg}{\mathrm{d}g}
\newcommand{\du}{\mathrm{d}u}
\newcommand{\ol}[1]{\overline{#1}}
\newcommand{\slz}{\SL_2(\Z)}
\newcommand{\slr}{\SL_2(\R)}
\newcommand{\slzh}{\slz \backslash \Hb}
\newcommand{\slzr}{\slz \backslash \slr}
\newcommand{\e}{\varepsilon}
\newcommand{\ph}{\varphi}
\newcommand{\tilda}{\widetilde}

\begin{document}

\title[Subconvexity Implies Effective Arithmetic Quantum Unique Ergodicity]{Subconvexity Implies Effective Quantum Unique Ergodicity for Hecke--Maa\texorpdfstring{\ss}{ß} Cusp Forms on $\slzr$}

\author{Ankit Bisain}

\address{Department of Mathematics, Massachusetts Institute of Technology, Cambridge, MA 02139, USA}

\email{\href{mailto:ankitb12@mit.edu}{ankitb12@mit.edu}}

\author{Peter Humphries}

\address{Department of Mathematics, University of Virginia, Charlottesville, VA 22904, USA}

\email{\href{mailto:pclhumphries@gmail.com}{pclhumphries@gmail.com}}

\urladdr{\href{https://sites.google.com/view/peterhumphries/}{https://sites.google.com/view/peterhumphries/}}

\author{Andrei Mandelshtam}

\address{Department of Mathematics, Stanford University, Stanford, CA 94305, USA}

\email{\href{mailto:mailto:andman@stanford.edu}{andman@stanford.edu}}

\author{Noah Walsh}

\address{Department of Mathematics, Massachusetts Institute of Technology, Cambridge, MA 02139, USA}

\email{\href{mailto:nwalsh21@mit.edu}{nwalsh21@mit.edu}}

\author{Xun Wang}

\address{Department of Mathematics, Northwestern University, Evanston, IL, 60208, USA}

\email{\href{mailto:stevenwang2029@u.northwestern.edu}{stevenwang2029@u.northwestern.edu}}



\begin{abstract}
It is a folklore result in arithmetic quantum chaos that quantum unique ergodicity on the modular surface with an effective rate of convergence follows from subconvex bounds for certain triple product $L$-functions. The physical space manifestation of this result, namely the equidistribution of mass of Hecke--Maa\ss{} cusp forms, was proven to follow from subconvexity by Watson, whereas the phase space manifestation of quantum unique ergodicity has only previously appeared in the literature for Eisenstein series via work of Jakobson. We detail the analogous phase space result for Hecke--Maa\ss{} cusp forms. The proof relies on the Watson--Ichino triple product formula together with a careful analysis of certain archimedean integrals of Whittaker functions.
\end{abstract}

\maketitle

\section{Introduction}

Quantum ergodicity, in its most general sense, originates from the study of quantum chaos. Loosely speaking, quantum ergodicity for a Riemannian manifold is the notion that almost all eigenfunctions of the Laplace--Beltrami operator equidistribute in the large eigenvalue limit. The foundational \emph{quantum ergodicity theorem} due to Shnirelman \cite{Shn74} proves quantum ergodicity for a compact Riemannian manifold with ergodic geodesic flow. In the language of quantum chaos, this can be seen as going from chaotic classical mechanics of a system to equidistribution of energy eigenstates of the system.

We begin with a brief introduction to the general case of quantum ergodicity. We then introduce \emph{arithmetic quantum chaos}, which will be the focus for the remainder of this paper. In the setting of arithmetic quantum chaos, notions such as quantum ergodicity are studied on manifolds with arithmetic structure, giving the eigenfunctions additional structure that is not present in the generic case. For surveys of the generic case of quantum ergodicity, see \cite{Ana10,DeB01,Dya22,Has11,Non13,Zel06,Zel10,Zel19}, while for surveys on arithmetic quantum chaos, see \cite{Mark05,Sar11}.

\subsection{Quantum Ergodicity}
\label{sec:QE}

\subsubsection{Classical Dynamics}

Let $(M,g)$ be a smooth compact oriented $n$-dimensional Riemannian manifold. The cotangent bundle $T^*M$ of the manifold $M$ consists of points $(x,\xi)$ with $x \in M$ and $\xi \in T_x^* M$, the space of tangent covectors at $x$. Associated to each point $x \in M$ and tangent vector $v \in T_x M$ is a unique geodesic $\gamma : \R \to M$ for which $\gamma(0) = x$ and $\gamma'(0) = v$, which gives rise to the geodesic flow $G^t(x,v) \coloneqq (\gamma(t),\gamma'(t))$ on the tangent bundle $T M$ and the corresponding geodesic flow $G_t(x,\xi)$ on the cotangent bundle $T^*M$ by identifying tangent covectors with the corresponding tangent vectors. Equivalently, the geodesic flow $G_t(x,\xi) = (x(t),\xi(t))$ is the Hamiltonian flow, obtained as solutions to the Hamiltonian equations
\[\frac{d}{dt} x_j(t) = \frac{\partial}{\partial \xi_j} H(x(t),\xi(t)), \qquad \frac{d}{dt} \xi_j(t) = -\frac{\partial}{\partial \xi_j} H(x(t),\xi(t)), \qquad j \in \{1,\ldots,n\},\]
of the Hamiltonian
\[H(x,\xi) \coloneqq \sum_{j,k = 1}^{n} g^{jk}(x) \xi_j \xi_k\]
on $T^*M$, where $g_{jk}(x) = g(\frac{\partial}{\partial x_j},\frac{\partial}{\partial x_k})$ denotes the Riemannian metric $g$, while $g^{jk}(x)$ denotes the entries of the inverse matrix. This geodesic flow preserves the cosphere bundle $S^*M$, which consists of points $(x,\xi) \in T^*M$ with $\xi$ of unit length.

One can think of the manifold $M$ as being \emph{physical space} that encodes the position of a point particle on $M$, while the cosphere bundle $S^*M$ is \emph{phase space} and encodes both the position and momentum of a point particle. The geodesic flow on $S^*M$ then encodes the position and momentum of a point particle over time and describes the \emph{classical dynamics} on $M$.

The metric $g$ induces probability measures $\mu$ and $\omega$ on $M$ and $S^*M$ respectively. The latter is called the \emph{Liouville measure} and is invariant under the geodesic flow. The geodesic flow on $S^*M$ is said to be \emph{ergodic} if for almost every $(x,\xi) \in S^*M$, the geodesic flow $G_t(x,\xi)$ equidistributes on $S^*M$ with respect to the Liouville measure. Ergodic geodesic flow demonstrates that the classical dynamics on $M$ are \emph{chaotic}.

\subsubsection{Quantum Dynamics}

These notions for classical dynamics on $M$ have quantum dynamical counterparts. Point particles are replaced by quantum particles with wave functions $\psi : M \times \R \to \C$. In place of $S^*M$, the space of states is instead wave functions $\psi$ that are square-integrable with respect to the volume measure $\mu$ on $M$; with $\psi$ $L^2$-normalized, the probability density $|\psi(x,t)|^2 \, \dmu(x)$ then encodes the probability that a quantum particle is located in a region at a given time $t$.

The classical dynamics $G_t$ are replaced by the quantum dynamics given by the evolution of a wave function over time as governed by Schr\"{o}dinger's equation
\[i \frac{\partial \psi}{\partial t} = -\Delta \psi.\]
Here
\[\Delta \coloneqq -\frac{1}{\sqrt{|\det g|}} \sum_{j,k = 1}^{n} \frac{\partial}{\partial x_j} g^{jk}(x) \sqrt{|\det g|} \frac{\partial}{\partial x_k}\]
denotes the Laplace--Beltrami operator on $M$, which we refer to as the Laplacian for the sake of brevity. This is a second order scalar linear partial differential operator, defined from the Riemannian metric $g$, that commutes with isometries of $M$; moreover, it is, up to scalar multiplication, the unique nontrivial such scalar linear partial differential operator of minimal order. The quantization of the Hamiltonian $H$ is the Laplacian $\Delta$, while the quantization of the geodesic flow $G_t$ acting on $S^*M$ is the quantum flow acting on $L^2(M)$ given by the Schr\"{o}dinger propagator $U^t \coloneqq e^{-it\Delta}$.

\emph{Stationary states} are $L^2$-normalized wave functions $\psi(x,t)$ that are separable, so that $\psi(x,t) = \varphi(x) \nu(t)$ for some functions $\varphi : M \to \C$ and $\nu : \R \to \C$, and are such that the probability densities $|\psi(x,t)|^2 \, \dmu(x)$ are independent of time, so that $\varphi \in L^2(M)$ and $\nu(t) = e^{-i\lambda t}$ for some $\lambda \in \R$. The functions $\varphi$ correspond to solutions of the eigenvalue problem
\[\Delta \varphi = \lambda \varphi;\]
that is, they are Laplacian eigenfunctions. For each eigenvalue, the corresponding eigenspace of Laplacian eigenfunctions is finite-dimensional, so that one can choose an orthonormal basis of each eigenspace. The union of these orthonormal bases then forms an orthonormal basis of the whole space $L^2(M)$, and the corresponding collection of eigenvalues forms a countable discrete set of nonnegative real numbers. We denote the set of Laplacian eigenfunctions by $(\ph_j)_{j \ge 1}$ and the corresponding Laplacian eigenvalues by $(\lambda_j)_{j \ge 1}$.

Associated to each Laplacian eigenfunction $\ph_j$ is its \emph{microlocal lift} $\omega_j$, alternatively known as a \emph{Wigner distribution}. The microlocal lift is a distribution on $S^*M$ of the measure corresponding to $\ph_j$ on the cosphere bundle $S^*M$, as defined in \cite[(2)]{Dya22}; it should be thought of as measuring the average value in phase space $S^*M$ of an observable $a \in C^{\infty}(S^* M)$ for a quantum particle with wave function $\varphi_j$. When acting on observables $a \in C^{\infty}(S^*M)$ that descend to functions on $M$, so that $a(x,\xi)$ is constant in $\xi$ the microlocal lift is simply the distribution $a(x,\xi) |\varphi_j(x)|^2 \, \dmu(x)$ on $M$.

\subsubsection{Quantum Ergodicity}

In \cite[Addendum]{Laz93}, Shnirelman proved that if the geodesic flow on $S^*M$ is ergodic with respect to the Liouville measure $\omega$, there exists a subsequence $(\ph_{j_k})_{k \ge 1}$ of the sequence of Laplacian eigenfunctions $(\ph_j)_{j \ge 1}$ of density $1$ (in the sense that $\#\{\lambda_{j_k} \leq \lambda\}/\#\{\lambda_k \leq \lambda\} \to 1$ as $\lambda \to \infty$) such that for all smooth functions $a$ on $M$,
\[\lim_{k \to \infty} \int_M a(x) \abs{\ph_{j_k}(x)}^2 \, \dmu(x) = \int_M a(x) \, \dmu(x).\]
That is, a density $1$ subsequence of the eigenfunctions \emph{equidistributes in physical space}. Shnirelman in fact proved the stronger statement that a density $1$ subsequence $(\omega_{j_k})_{k \geq 1}$ of the sequence of microlocal lifts \emph{equidistributes in phase space} in the sense that it approaches the Liouville measure on $S^*M$. That is, for any smooth function $a$ on $S^*M$,
\[\lim_{k \to \infty} \int_{S^*M} a(x,\xi) \,\domega_{j_k}(x,\xi) = \int_{S^*M} a(x,\xi) \,\domega(x,\xi).\]
This property is known as \emph{quantum ergodicity}. An outline of a proof of the quantum ergodicity theorem similar to Shnirelman's original proof can be found in \cite[Section 2]{Dya22}. Shnirelman's proof was first announced in \cite{Shn74} and independent proofs were obtained by Zelditch \cite{Zel87} and Colin de Verdi\`{e}re \cite{Col85}. Quantum ergodicity demonstrates that the quantum dynamics on $M$ are \emph{chaotic}. The fact that quantum ergodicity follows from the assumption that the geodesic flow on $S^*M$ is ergodic can be viewed as showing that chaotic classical dynamics imply chaotic quantum dynamics.

Quantum \emph{unique} ergodicity (QUE) in physical space is the property that $(\ph_j)_{j \ge 1}$ satisfies 
\[\lim_{j \to \infty} \int_M a(x) \abs{\ph_{j}(x)}^2 \,\dmu(x) = \int_M a(x)\, \dmu(x)\]
for all smooth functions $a$ on $M$. Equivalently, QUE in physical space is the property that the whole sequence of eigenfunctions equidistributes in physical space $M$. The notion of QUE has a natural generalization to phase space $S^{\ast} M$: quantum unique ergodicity in phase space refers to the property of $(\ph_j)_{j \geq 1}$ satisfying
\begin{equation}
\lim_{j \to \infty} \int_{S^*M} a(x,\xi) \,\domega_{j}(x,\xi) = \int_{S^*M} a(x,\xi) \,\domega(x,\xi)
\end{equation}
for all smooth functions $a$ on $S^{\ast} M$. Henceforth, QUE will refer to quantum unique ergodicity on phase space unless otherwise noted.

It was established by Hassell in \cite[Theorem 1]{Has10} that there exist compact Riemannian manifolds for which the geodesic flow is ergodic and yet not all eigenfunctions equidistribute. Namely, Hassell showed that QUE does not hold for a large family of stadium domains.\footnote{For manifolds with boundary, the geodesic flow is replaced by the \emph{billiard flow}, where trajectories bounce off of the boundary.} However, in many cases, it is still believed that QUE should hold. In particular, it was conjectured by Rudnick and Sarnak \cite[Conjecture]{RS94} that QUE holds when $(M,g)$ is a compact hyperbolic surface or more generally a negatively curved compact manifold.

\subsection{Quantum Unique Ergodicity for Arithmetic Surfaces}

For most hyperbolic surfaces, QUE is far from proven. However, this conjecture is better understood in the case where $(M,g)$ is an arithmetic hyperbolic surface.

Let $\Hb \coloneqq \{z = x+iy \in \C : y \in \R_+\}$ denote the upper half-plane with area measure $\dmu(z) \coloneqq \frac{\dx \, \dy}{y^2}$ and Laplacian $\Delta \coloneqq -y^2(\frac{\partial^2}{\partial x^2} + \frac{\partial^2}{\partial y^2})$ coming from the standard hyperbolic metric $\ds^2 \coloneqq \frac{\dx^2+\dy^2}{y^2}$. Recall that $\slr$ acts transitively on $\Hb$ via M\"{o}bius transformations, namely
\[g \cdot z \coloneqq \frac{az + b}{cz + d} \quad \text{for $g = \begin{pmatrix}a & b \\ c & d \end{pmatrix} \in \SL_2(\R)$ and $z \in \Hb$.}\]
If $\Gamma \subset \slr$ is an arithmetic subgroup (in the sense of \cite[Chapter 5]{Kat92}), such as a congruence subgroup of $\slz$, the quotient space $\Gamma\backslash \Hb$ is an \emph{arithmetic hyperbolic surface}. These surfaces are not necessarily compact, but have finite area, allowing the necessary notions to be defined. In particular, $\Gamma \backslash \Hb$ has finite area (with respect to $\dmu$) given by $\frac{\pi}{3}[\slz : \Gamma]$ when $\Gamma$ is a finite-index subgroup of the modular group $\slz$. Note that in general, the area measure $d\mu(z)$ is \emph{not} a probability measure on $\Gamma \backslash \Hb$, which differs from the normalization of the probability measure $\mu$ on $M$ in Section \ref{sec:QE}.

The study of QUE on arithmetic surfaces is aided via the presence of \emph{Hecke operators} (see \eqref{eqn:Heckeoperator} below). The Hecke operators on a given arithmetic hyperbolic surface are a sequence $T_1,T_2,\ldots$ of self-adjoint operators on the space of square-integrable functions on the surface; they are a nonarchimedean analogue of the Laplace--Beltrami operator. It is known that Hecke operators commute with each other and with the hyperbolic Laplacian $\Delta$. We may therefore simultaneously diagonalize the space of Maa{\ss} cusp forms (nonconstant Laplacian eigenfunctions occurring in the discrete spectrum of the Laplacian) with respect to the Hecke operators, obtaining a basis of \emph{Hecke--Maa{\ss} cusp forms}, which are simultaneous eigenfunctions of both the Laplacian and of all the Hecke operators. Due to the additional structure given from the Hecke operators, stronger results regarding QUE are known for such Hecke eigenbases.

Henceforth, we focus on the case where $\Gamma = \slz$ and $M=\slzh$ is the modular surface. This surface is not compact, as it has a cusp at $i\infty$. Its cosphere bundle $S^{\ast} M$ may be identified with the quotient space $\slzr$, while the microlocal lift $\omega_j$ of a Laplacian eigenfunction can be explicitly expressed in terms of linear combinations of raised and lowered Laplacian eigenfunctions, as we explicate further in Section \ref{sec:functions}. Its Laplacian eigenfunctions can be split into two classes. There is a discrete spectrum, which, besides constant functions, arises from nonconstant Laplacian eigenfunctions $\ph_j$ called \emph{Maa{\ss} cusp forms} corresponding to a nondecreasing sequence of positive eigenvalues $\lambda_j$; the \emph{cuspidality condition} is precisely the condition that
\[\int_{0}^{1} \varphi_j(x + iy) \, \dx = 0\]
for all $y \in \R_+$. Because $M$ is noncompact, there is also a continuous spectrum, with eigenfunctions coming from real-analytic Eisenstein series $E\left(z,\frac{1}{2}+it\right)$ with eigenvalues $\frac{1}{4}+t^2$. We discuss Hecke--Maa\ss{} cusp forms and Eisenstein series in further detail in Section \ref{sec:functions} (cf.\ \cite[Section 4]{DFI02}).

It is a seminal result of Lindenstrauss \cite[Theorem 1.4]{Lin06} that on a (possibly noncompact) arithmetic hyperbolic surface, for a Hecke eigenbasis, any limit (in the weak-* topology) of a subsequence of the measures $\omega_j$ is a nonnegative multiple of the Liouville measure $\omega$. When the surface is compact, this limit must be the Liouville measure itself, proving QUE for compact arithmetic hyperbolic surfaces. See \cite[Section 3]{Sar11} for more discussion of the relevant work and progress in the arithmetic case.

On the (noncompact) modular surface, equidistribution for the continuous spectrum was established in physical space by Luo and Sarnak \cite[Theorem \ref{thm:mainthm}]{LS95}, and later in phase space by Jakobson \cite[Theorem 1]{Jak94}. Since the modular surface is noncompact, the work of Lindenstrauss does not establish QUE for this surface, as there is possibility of mass escaping to the cusp. This possibility was eliminated by Soundararajan \cite{Sou10}, establishing QUE for Hecke--Maa{\ss} cusp forms on the modular surface. However, this resolution of QUE for Hecke--Maa{\ss} cusp forms leaves unresolved the problem of determining the \emph{rate} of equidistribution.

In \cite[Theorem 2]{Jak97}, Jakobson proves that the measures $\omega_j$ converge to $\omega$ in an averaged sense with an effective rate of averaged equidistribution. Precisely, Jakobson proves that if $a$ is an element of the space $C_{c,K}^{\infty}(S^{\ast} M)$ consisting of finite linear combinations of smooth compactly supported functions of even weight (as described in \eqref{eqn:weighttransformation} below), then
\begin{equation}
\label{eqn:Jakobsontheorem1}
\sum_{\lambda_j \leq \lambda}\left|\int_{S^*M} a(z, \theta) \, \domega_j(z, \theta) - \int_{S^*M} a(z, \theta) \,\domega(z, \theta)\right|^2 \ll_{a,\e} \lambda^{\frac{1}{2} + \e}.
\end{equation}
As Weyl's law implies that the number of eigenvalues below $\lambda$ is asymptotic to $\frac{\lambda}{12}$ \cite[Theorem 2]{Ris04}, this gives an averaged bound of $\lambda^{-1/2+\e}$ on each summand. This bound generalized an earlier result of Luo and Sarnak \cite[Theorem 1.2]{LS95}, which essentially gave the analogous average bound in physical space. Luo and Sarnak also remark that the best possible individual bound for each summand in \eqref{eqn:Jakobsontheorem1} is of size $\lambda_j^{-1/2}$. To see why this is true, we recall that it was established by Sarnak and Zhao \cite[Theorem 1.1]{SZ19} that 
\[\sum_{\lambda_j \leq \lambda} \left|\int_{S^*M} a(z, \theta) \, \domega_j(z, \theta) - \int_{S^*M} a(z, \theta) \,\domega(z, \theta)\right|^2 \sim Q(a,a) \lambda^{\frac{1}{2}},\]
where $Q(a,b)$ is a fixed sesquilinear form on $C_{c,K}^\infty(\slzr) \times C_{c,K}^\infty(\slzr)$. It follows that if
\[\max_{\lambda_j \leq \lambda} \left|\int_{S^*M} a(z, \theta) \, \domega_j(z, \theta) - \int_{S^*M} a(z, \theta) \,\domega(z, \theta)\right| \leq C\]
for some nonnegative constant $C$, then
\[\sum_{\lambda_j \leq \lambda}\left|\int_{S^*M} a(z, \theta) \, \domega_j(z, \theta) - \int_{S^*M} a(z, \theta) \,\domega(z, \theta)\right|^2 \leq C^2 \left(\frac{\lambda}{12}+o(\lambda)\right),\]
which are contradictory statements unless $C \gg \lambda^{-1/4}$.

\subsection{Results}

Our goal is to prove bounds for the \emph{individual} terms
\[\int_{S^*M} a(z, \theta) \, \domega_j(z, \theta) - \int_{S^*M} a(z, \theta) \,\domega(z, \theta).\]
These bounds are contingent on bounds for certain $L$-functions. In \cite[Theorem 3]{Wat08}, Watson establishes the following formula for integrals of products of Hecke--Maa{\ss} cusp forms $\varphi_j$, whose precise definitions we give in Section \ref{sec:functions}: there exists a nonnegative absolute constant $C$ such that
\[\left| \int_{M} \ph_{j_1}(z)\ph_{j_2}(z) \ph_{j_3}(z) \, \dmu(z) \right|^2 = C \frac{\Lambda\left(\frac{1}{2}, \ph_{j_1} \otimes \ph_{j_2} \otimes \ph_{j_3}\right)}{\Lambda(1,\ad\ph_{j_1})\Lambda(1,\ad\ph_{j_2})\Lambda(1,\ad\ph_{j_3})}.\]
Here the terms on the right-hand side are completed $L$-functions whose definitions are given in Section \ref{sec:Lfunctions}. The Lindel\"{o}f hypothesis for such $L$-functions (itself a consequence of the generalized Riemann hypothesis) would then imply sufficiently strong upper bounds in order to prove the uniform version of Luo and Sarnak's physical space result \cite[Theorem 1.2]{LS95}. In particular, for any $a \in C^\infty_c(M)$, we would have that
\[\int_M a(z) \abs{\ph_j(z)}^2 \, \dmu(z) - \int_M a(z) \, \dmu(z) \ll_{a,\e} \lambda_j^{-\frac{1}{4} + \e}\]
under the assumption of the conjectural bound $L\left(\half,\ph_{j_1} \otimes \ph_{j_1} \otimes \ph_{j_3}\right) \ll_{\ph_{j_3},\e} \lambda_{j_1}^{\e}$ (cf.\ \cite[Corollary 1]{Wat08} and \cite[Proposition 1.5]{You16}). More generally, any effective subconvex bound of the form $L\left(\half,\ph_{j_1} \otimes \ph_{j_1} \otimes \ph_{j_3}\right) \ll_{\ph_{j_3}} \lambda_{j_1}^{1/2 - 2\delta}$ would provide the above statement with weaker error term of the form $O_{a}(\lambda_j^{-\delta} \log \lambda_j)$. In this paper, we prove the strengthening of this physical space statement to phase space.

\begin{theorem}
\label{thm:mainthm}
Suppose that there exist constants $\delta > 0$ and $A > 0$ such that for any Hecke--Maa\ss{} cusp forms $\varphi_1,\varphi_2$ with Laplacian eigenvalues $\lambda_1,\lambda_2$, any $t \in \R$, and any holomorphic Hecke cusp form $F$, we have the subconvex bounds
\begin{align}
\label{eqn:subconvex1}
L\left(\frac{1}{2},\ad \varphi_1 \otimes \varphi_2\right) & \ll \lambda_1^{\frac{1}{2} - 2\delta} \lambda_2^A,	\\
\label{eqn:subconvex2}
L\left(\frac{1}{2} + it,\ad \varphi_1\right) & \ll \lambda_1^{\frac{1}{4} - \delta} (1 + |t|)^A,	\\
\label{eqn:subconvex3}
L\left(\frac{1}{2},\ad \varphi_1 \otimes F\right) & \ll_F \lambda_1^{\frac{1}{2} - 2\delta}.
\end{align}
Then for any $a \in C_{c,K}^\infty(S^*M)$, we have that
\begin{equation}
\label{eqn:mainthm}
\int_{S^*M} a(z,\theta) \, \domega_j(z,\theta)-\int_{S^*M} a(z,\theta)\, \domega(z,\theta) \ll_a \lambda_j^{-\delta} \log \lambda_j.
\end{equation}
In particular, assuming the generalized Lindel\"{o}f hypothesis, we have that
\[\int_{S^*M} a(z,\theta) \, \domega_j(z,\theta)-\int_{S^*M} a(z,\theta)\, \domega(z,\theta) \ll_{a,\e}\lambda_j^{-\frac{1}{4} + \e}.\]
\end{theorem}

\begin{remark}
The method of proof yields explicit dependence on $a$ in these error terms in terms of Sobolev norms of $a$; see \eqref{eqn:Sobolev}.
\end{remark}

The subconvex bounds \eqref{eqn:subconvex1}, \eqref{eqn:subconvex2}, and \eqref{eqn:subconvex3} are \emph{hypotheses} in Theorem \ref{thm:mainthm}: it is not yet known that these subconvex bounds hold. While subconvex bounds are known for certain other families of $L$-functions (see, for example, \cite{MV10,Nel21}), unconditional proofs of the desired subconvex bounds \eqref{eqn:subconvex1}, \eqref{eqn:subconvex2}, and \eqref{eqn:subconvex3} currently remain elusive. A chief obstacle towards proving these bounds is the so-called \emph{conductor dropping phenomenon}, as discussed in \cite{KY23}.

Theorem \ref{thm:mainthm} is folklore (see, for example, \cite[p.\ 1156]{SZ19}), though no detailed proof exists in the literature. The method of proof is known to experts; the analogue of QUE for Bianchi manifolds (i.e.\ arithmetic quotients of $\Hb^3 = \SL_2(\C)/\mathrm{SU}(2)$), for example, has been shown by Marshall to follow from subconvexity for triple product $L$-functions \cite[Theorem 3]{Mars14}, and the proof that we give for the modular surface is by the same general strategy. To explicate all the details, one needs the full strength of the Watson--Ichino triple product formula as in \cite[Theorem 3]{Wat08} and \cite[Theorem 1.1]{Ich08}. Coupling this with a lemma of Michel and Venkatesh \cite[Lemma 3.4.2]{MV10} (cf.\ \cite[Lemma 5]{SZ19}), we show that certain triple products of automorphic forms on $\slzr$ can be expressed in terms of a product of central values of $L$-functions and certain archimedean integrals of Whittaker functions; the latter can in turn be related to gamma functions and hypergeometric functions.

Finally, we take this opportunity to observe that Jakobson's treatment of QUE for Eisenstein series in \cite{Jak94} is incomplete; in particular, the case where the test function is a shifted holomorphic or antiholomorphic Hecke cusp form is missing. We supply the omitted computations in Section \ref{sec:fixing}.

\subsection{Friedrichs Symmetrization}

We end the discussion of our results by explaining how our results are valid not only for the Wigner distribution $\omega_j$, which need not be a positive distribution, but also for the \emph{Friedrichs symmetrization} $\omega_j^F$ defined in \eqref{eqn:Friedrichs} below, which \emph{is} a positive distribution. The microlocal lifts $\omega_j$ of Hecke--Maa\ss{} cusp forms on the modular surface that we work with in this paper are the Wigner distributions given by
\[\domega_j(z,\theta) \coloneqq \ph_j(z) \ol{u_j(z, \theta)} \, \domega(z,\theta), \qquad u_j(z,\theta) \coloneqq\frac{3}{\pi} \sum_{k = -\infty}^{\infty}\varphi_{j, k}(z)e^{2ki\theta},\]
as defined in \cite[(1.18)]{Zel91}. Here the convergence is in distribution and $\domega$ is the (unnormalized) Liouville measure, given by $\frac{\dx \, \dy \, \dtheta}{2\pi y^2}$ on $\slzr = S^*M$, where we identify $g \in \slr$ with $(x,y,\theta) \in \R \times \R_+ \times [0,2\pi)$ via the Iwasawa decomposition (cf.\ \eqref{eqn:iwasawa} below). The functions $\ph_{j,k}$ are the $L^2$-normalized shifted Hecke--Maa{\ss} forms of weight $2k$ obtained from $\ph_j$ by raising or lowering operators, as defined in Section \ref{sec:functions}; for their Fourier expansions, see Section \ref{sec:fourier}. 

We recall that a positive distribution $T$ on a normed space $V$ over $\C$ is a bounded linear functional $T: V \to \C$ such that $T(v) \geq 0$ for all $v \in V$. In general, the Wigner distribution $\domega_j$ need not be a positive distribution on $C_c^\infty(\slzr)$. To convert $\domega_j$ into a positive distribution, we define for $a \in C_c^\infty(\slzr)$ the pairing
\[\left\langle a , \domega_j \right\rangle = \int_{S^*M} a(z,\theta) \, \domega_j(z,\theta) \coloneqq \lim_{K \to \infty} \int_{\slzr} a(z,\theta) \varphi_j(z) \sum_{k = -K}^K \ol{\varphi_{j,k}(z) e^{2ki\theta}} \, \domega(z,\theta).\]
We now define a new distribution $\domega_j^F$, the \emph{Friedrichs symmetrization} of $\domega_j$, via
\begin{equation}
\label{eqn:Friedrichs}
\left\langle a, \domega_j^F \right\rangle \coloneqq \left\langle a^F, \domega_j \right\rangle,
\end{equation}
where the function $a^F \in C_c^{\infty}(\slzr)$ is the Friedrichs symmetrization of $a$; for its explicit construction, see \cite[Proposition 2.3]{Zel87}. In particular, it was established in \cite[Proposition 2.3]{Zel87} that $\domega_j^F $ is a positive distribution\footnote{Lindenstrauss \cite[Corollary 3.2]{Lin01} constructs an alternate positive distribution that has a similar effect: for each $N \in \N$, Lindenstrauss defines the positive distribution $\domega_j^N(z,\theta) \coloneqq \frac{3}{\pi} \frac{1}{2N + 1} \left| \sum_{k = -N}^{N} \varphi_{j,k}(z) e^{2ki\theta}\right|^2 \, \domega(z,\theta)$. For $N \sim \lambda_j^{1/4}$, this satisfies $\langle a, \domega_j^N \rangle - \langle a, \domega_j \rangle \ll_{a,\e} \lambda_j^{-1/4+\varepsilon}$.}, while it was established in \cite[Proposition 3.8]{Zel91} that
\[\left\langle a, \domega_j^F \right\rangle - \left\langle a, \domega_j \right\rangle \ll_{a,\e} \lambda_j^{-\frac{1}{2}+\varepsilon}.\]
Combined with Theorem \ref{thm:mainthm}, we see that in specific scenarios where one needs to deal with positive distributions, it suffices to work with the Wigner distribution $\domega_j$.

\section*{Acknowledgements}

The authors participated in, and conducted this research through, the 2023 UVA REU in Number Theory. They are grateful for the support of grants from Jane Street Capital, the National Science Foundation (DMS-2002265 and DMS-2147273), the National Security Agency (H98230-23-1-0016), and the Templeton World Charity Foundation. The second author was additionally supported by the National Science Foundation (grant DMS-2302079) and by the Simons Foundation (award 965056).

The authors would like to thank Maximiliano Sanchez Garza for his guidance, Asaf Katz and Jesse Thorner for useful feedback, and the anonymous referees for their thorough readings of this paper. Finally, they are grateful to Ken Ono for organizing the 2023 UVA REU.

\section{Proof Outline}

On a broad scale, our proof strategy follows the proof of equidistribution of Eisenstein series in phase space from \cite{Jak94}, which we now outline. We will also make reference to a few objects that we have not yet defined; namely, we use $(x,y,\theta)$ coordinates on $S^*M$ given by \eqref{eqn:iwasawa}, $L$-functions that we explain in Section \ref{sec:Lfunctions}, and various types of functions on $M$ all defined in Section \ref{sec:functions}.

In our paper, we extend the probability measure $|\varphi_j|^2\, \dmu$ to its microlocal lift $\domega_j$ on $S^*M$ for a Hecke--Maa{\ss} cusp form $\varphi_j$ with Laplacian eigenvalue $\lambda_j$. The work of Jakobson \cite{Jak94} solves a similar problem: Jakobson proves the analogous result for the extension of the Radon measure $|E(\cdot, \frac{1}{2}+it)|^2\, \dmu$ to its microlocal lift $\dmu_t$. Jakobson's method for bounding integrals of the form $\int a\, \dmu_t$ is to consider only functions $a$ appearing in an orthonormal basis of $L^2(S^*M)$. Namely, Jakobson computes the integral for constant functions, shifted Hecke--Maa{\ss} cusp forms, shifted holomorphic or antiholomorphic Hecke cusp forms\footnote{As mentioned previously, Jakobson only treats unshifted holomorphic Hecke cusp forms and neglects to deal with the more general case of shifted holomorphic or antiholomorphic Hecke cusp forms. We complete Jakobson's proof by dealing with this untreated general case in Section \ref{sec:fixing}.}, and weighted Eisenstein series. He then bounds $\int a\, \dmu_t$ for general smooth, compactly supported $a$ on $S^*M$ by approximating them using this basis.

To bound $\int a\, \dmu_t$, Jakobson uses the coordinates $(x, y, \theta)$ on $S^*M$ and proceeds to integrate over $\theta$, which reduces the problem to computing integrals over $M$. These integrals can readily be evaluated using the key fact that they involve Eisenstein series. An Eisenstein series can be written by a sum over $\Gamma_\infty\backslash\slz$, where $\Gamma_\infty \subset \slz$ is the stabilizer of the cusp $i\infty$, in such a way that the integral can be unfolded to one over the fundamental domain $\{x + iy \in \Hb : x \in [0,1]\}$ for $\Gamma_{\infty}\backslash\Hb$. Jakobson then inserts the Fourier--Whittaker expansion of each function in the integrand and subsequently directly evaluates the integral over $x \in [0,1]$. One is left with an expression involving central values of $L$-functions related to the test functions and an integral over $y \in \R_+$ of Whittaker functions. This remaining integral can be expressed in terms of hypergeometric functions and subsequently bounded using Stirling's formula.

Our paper follows a similar reduction of integrals, using the same orthonormal basis. In particular, we must show that the constant term contributes the main term in Theorem \ref{thm:mainthm}, while the contribution from integrating against shifted Hecke--Maa{\ss} cusp forms, shifted holomorphic or antiholomorphic Hecke cusp forms, and shifted Eisenstein series are $O(\lambda_j^{-\delta} \log \lambda_j)$ as $j \to \infty$. We now outline how we evaluate each type of integral.
\begin{itemize}
\item The constant case is trivial, and contributes to the main term in Theorem \ref{thm:mainthm}.
\item The weighted Eisenstein series case can be computed with an unfolding technique analogous to the previously discussed computations in \cite{Jak94}. Computing this integral gives a product of a central value of an $L$-function and an expression involving gamma functions and hypergeometric functions.
\item For the remaining two cases, namely shifted Hecke--Maa{\ss} cusp forms and shifted holomorphic or antiholomorphic Hecke cusp forms, the unfolding trick does not apply to the integrals of interest since they do not involve an Eisenstein series. Instead, we use the \emph{Watson--Ichino triple product formula} \cite{Ich08,Wat08}. This formula allows us to write the square of the absolute value of the integral as a product of a central value of an $L$-function and the square of the absolute value of an integral of Whittaker functions. The latter integral can again be explicitly computed to obtain an expression in terms of hypergeometric functions.
\end{itemize}
We then bound all hypergeometric functions using Stirling's formula, while we invoke our assumption of subconvexity to bound central values of $L$-functions, which yields Theorem \ref{thm:mainthm}.

\section{Preliminaries}

\subsection{Differential Operators on \texorpdfstring{$\slzr$}{SL\texttwoinferior(Z)\textbackslash SL\texttwoinferior(R)} and \texorpdfstring{$\slzh$}{SL\texttwoinferior(Z)\textbackslash H}}

We begin by describing differential operators acting on $\slzr$ and $\slzr$. Useful references for these include \cite[Chapter 2]{Bum97}, \cite[Chapter 1]{Iwa02}, \cite{Lan85}, and \cite{Roe66}.

In coordinates $z = x + iy \in \Hb$, the Laplacian on $\slzh$ is given by $\Delta \coloneqq -y^2(\frac{\partial^2}{\partial x^2} + \frac{\partial^2}{\partial y^2})$, and the area measure is given by $\dmu(z) \coloneqq \frac{\dx\, \dy}{y^2}$, giving this space volume $\frac{\pi}{3}$. The unnormalized Liouville measure on the unit cotangent bundle $S^*M = \slzr$ is given by $\domega(z,\theta) \coloneqq \frac{\dmu(z)\, \dtheta}{2\pi}$, which also gives this space volume $\frac{\pi}{3}$. Here we identify points on $S^*M$ with points on $\slzr$ using the Iwasawa decomposition
\begin{equation}
\label{eqn:iwasawa}
g = 
\begin{pmatrix}
1 & x\\
0 & 1
\end{pmatrix}
\begin{pmatrix}
y^{1/2} & 0\\
0 & y^{-1/2}
\end{pmatrix}
\begin{pmatrix}
\cos\theta & \sin\theta\\
-\sin\theta & \cos\theta
\end{pmatrix}
\end{equation}
for elements $g \in \slr$, where $x \in \R$, $y \in \R_+$, and $\theta \in [0,2\pi)$.

A function $\Phi : \slzr \to \C$ is of weight $2k$ for some $k \in \Z$ if it satisfies
\begin{equation}
\label{eqn:weighttransformation}
\Phi(z,\theta + \theta') = e^{2ki\theta'} \Phi(z,\theta)
\end{equation}
for all $z \in \Hb$ and $\theta,\theta' \in \R$. We have an inner product on $\slzr$ defined by
\[\left\langle \Phi_1, \Phi_2 \right\rangle \coloneqq \int_{\slzr} \Phi_1(z,\theta) \ol{\Phi_2(z,\theta)} \, \domega(z,\theta)\]
for $\Phi_1,\Phi_2 \in L^2(\slzr)$.

Similarly, a function $f : \Hb \to \C$ satisfying the automorphy condition
\begin{equation}
\label{eqn:automorphy}
f\left(\frac{az + b}{cz + d}\right) = \left(\frac{cz + d}{|cz + d|}\right)^{2k} f(z)
\end{equation}
for all $\begin{psmallmatrix} a & b \\ c & d \end{psmallmatrix} \in \slz$ is said to be of weight $2k$. A weight $2k$ function $f : \Hb \to \C$ lifts to a weight $2k$ function $\Phi : \slzr \to \C$ via the map $f \mapsto \Phi$ given by $\Phi(z,\theta) \coloneqq f(z) e^{2ki\theta}$. Equivalently, we define for $g = \begin{psmallmatrix} a & b \\ c & d \end{psmallmatrix} \in \SL_2(\R)$ and $z \in \Hb$ the $j$-factor
\[j_g(z) \coloneqq \frac{cz + d}{|cz + d|},\]
so that for $g \in \SL_2(\R)$ given by \eqref{eqn:iwasawa}, we have that $\Phi(z,\theta) \coloneqq j_g(i)^{-k} f(g \cdot i)$. We have an inner product on weight $2k$ functions $f_1,f_2 : \Hb \to \C$ defined by
\[\left\langle f_1, f_2 \right\rangle \coloneqq \int_{\slzh} f_1(z) \ol{f_2(z)} \, \dmu(z).\]

The $\slr$-invariant extension of $\Delta$ from functions on $\Hb$ to functions on $\slr$ is given by the Casimir operator
\[\Omega \coloneqq -y^2\left(\frac{\partial^2}{\partial x^2} + \frac{\partial^2}{\partial y^2}\right) + y\frac{\partial^2}{\partial x \partial \theta}.\]
We also have raising and lowering operators
\[R \coloneqq e^{2i\theta} iy \left(\frac{\partial}{\partial x} - i\frac{\partial}{\partial y}\right) - e^{2i\theta} \frac{i}{2} \frac{\partial}{\partial \theta}, \qquad L \coloneqq -e^{-2i\theta} iy \left(\frac{\partial}{\partial x} + i\frac{\partial}{\partial y}\right) + e^{-2i\theta} \frac{i}{2} \frac{\partial}{\partial \theta}.\]
The operators $\Omega,R,L$ are initially defined on the space $C^{\infty}(\slzr)$ of smooth functions. The raising and lowering operators $R$ and $L$ map weight $2k$ eigenfunctions of $\Omega$ to weight $2k+2$ and $2k-2$ eigenfunctions of $\Omega$ respectively. On the space $C^{\infty}(\slzr) \cap L^2(\slzr)$, these operators are such that $-L$ is adjoint to $R$, while $\Omega$ is self-adjoint, so that for all smooth square-integrable functions $\Phi_1,\Phi_2 : \slzr \to \C$,
\[\langle R \Phi_1,\Phi_2 \rangle = -\langle \Phi_1,L \Phi_2\rangle, \qquad \langle \Omega \Phi_1, \Phi_2\rangle = \langle \Phi_1, \Omega \Phi_2\rangle.\]
The Casimir operator $\Omega$ admits a canonical self-adjoint extension, the \emph{Friedrichs extension}, to $L^2(\slzr)$ (cf.\ \cite[Theorem A.3]{Iwa02}).

These operators descend to $\slzh$. Considering the action of $\Omega$ on weight $2k$ functions on $\Hb$, we have the corresponding weight $2k$ Laplacian on $\Hb$ given by
\[\Delta_{2k} \coloneqq \Delta + 2kiy \frac{\partial}{\partial x},\]
which preserves the weight of a weight $2k$ function on $\Hb$. Similarly, $R$ and $L$ become the raising and lowering operators
\[R_{2k} \coloneqq iy\left(\frac{\partial}{\partial x}-i\frac{\partial}{\partial y} \right)+k, \qquad L_{2k} \coloneqq -iy\left(\frac{\partial}{\partial x}+i\frac{\partial}{\partial y} \right)-k.\]
The raising operator $R_{2k}$ maps weight $2k$ functions to weight $2k + 2$ functions, whereas the lowering operator $L_{2k}$ maps weight $2k$ functions to weight $2k - 2$ functions. In particular, the raising and lowering operators map eigenfunctions of $\Delta_{2k}$ to eigenfunctions of $\Delta_{2k+2}$ and $\Delta_{2k-2}$ respectively. The inner product on weight $2k$ functions on $\Hb$ is such that $-L_{2k + 2}$ is adjoint to $R_{2k}$, so that
\[\langle R_{2k} f_1,f_2 \rangle = -\langle f_1,L_{2k + 2} f_2\rangle\]
for all weight $2k$ square-integrable functions $f_1 : \Hb \to \C$ and weight $2k + 2$ square-integrable functions $f_2 : \Hb \to \C$.

\subsection{Eigenfunctions of the Laplacian}
\label{sec:functions}

Next, we describe the eigenfunctions of $\Delta_{2k}$ on $\slzh$. Useful references for these include \cite[Chapter 2]{Bum97}, \cite[Section 4]{DFI02}, and \cite{Roe66}.

For any $k \in \Z$, there are up to four classes of eigenfunctions of $\Delta_{2k}$ of weight $2k$. Each of these is an eigenfunction of the $n$-th Hecke operator $T_n$ for each $n\in\N$, where $T_n$ acts on functions $f:\Hb\to\C$ via
\begin{equation}
\label{eqn:Heckeoperator}
(T_nf)(z) \coloneqq \frac{1}{\sqrt{n}}\sum_{ad=n}\sum_{b=1}^d f\left(\frac{az+b}{d} \right).
\end{equation}
Each of these eigenfunctions of $\Delta_{2k}$ also lifts to a function on $\slzr$ that is an eigenfunction of $\Omega$.

\begin{itemize}
\item When $k=0$, we have constant functions.
\item When $k \geq 0$, we have shifted Maa{\ss} cusp forms of weight $2k$ given by $R_{2k-2}\cdots R_0 \ph_j$, where $\ph_j$ is a Hecke--Maa{\ss} cusp form of weight $0$ with $j$-th Laplacian eigenvalue $\lambda_j$ (ordered by size). Similarly, when $k \le 0$ we have forms of weight $2k$ given by $L_{2k+2}L_{2k+4}\cdots L_{0}\ph_j$. Any weight $0$ form $\varphi_j$ can be written as a sum of an even part and an odd part with the same Laplacian and Hecke eigenvalues, so we may additionally assume that $\varphi_j$ is either even, so that $\varphi_j(-\overline{z}) = \varphi_j(z)$, or odd, so that $\varphi_j(-\overline{z}) = -\varphi_j(z)$. We let $\kappa_j \in \{0,1\}$ be such that $\kappa_j$ is $0$ if $\varphi_j$ is even and $\kappa_j$ is $1$ if $\varphi_j$ is odd; the \emph{parity} of $\varphi_j$ is then defined to be $\epsilon_j = (-1)^{\kappa_j}$, so that
\begin{equation}
\label{eqn:paritydefeq}
\varphi_j(-\overline{z}) = (-1)^{\kappa_j} \varphi_j(z) = \epsilon_j \varphi_j(z).
\end{equation}
The \emph{spectral parameter} $r_j \in [0,\infty) \cup i(0,\frac{1}{2})$ satisfies $\lambda_j = \frac 14 + r_j^2$; since the Selberg eigenvalue conjecture is known for $\slzh$, $r_j$ must be real and positive (with the smallest spectral parameter being $r_1 \approx 9.534$). Once $L^2$-normalized with respect to the measure $\dmu$ on $\slzh$, the eigenfunctions $\varphi_j$ yield probability measures $\dmu_j = |\varphi_j|^2\, \dmu$ on $\slzh$. The corresponding $L^2$-normalized shifted Hecke--Maa\ss{} cusp forms of weight $2k$ are given by
\[\varphi_{j, k} \coloneqq \begin{dcases*}
\frac{\gamer{\frac{1}{2}+ir_j}}{\gamer{\frac{1}{2}+k+ir_j}}R_{2k-2}\cdots R_2R_0\varphi_j & for $k \geq 0$, \\
\frac{\gamer{\frac{1}{2}+ir_j}}{\gamer{\frac{1}{2}-k+ir_j}} L_{2k+2}\cdots L_{-2}L_0\varphi_j & for $k \leq 0$,
\end{dcases*}\]
where the $L^2$-normalization of each $\varphi_{j,k}$ follows from the fact that $-L_{2k + 2}$ is adjoint to $R_{2k}$ (cf.\ \cite[Corollary 4.4]{DFI02}). The associated lift to $\slzr$ is the function $\Phi_{j,k}(z,\theta) \coloneqq \varphi_{j,k}(z) e^{2ki\theta}$, which is an eigenfunction of the Casimir operator $\Omega$ with eigenvalue $\lambda_j$.
\item When $\ell \geq 1$, let $F$ be a holomorphic Hecke cusp form of weight $2\ell$; there are finitely many such cusp forms, and we denote the set of such holomorphic Hecke cusp forms by $\mathcal{H}_{\ell}$. We define a corresponding weight $2\ell$ function $f(z)=y^\ell F(z)$, which is automorphic of weight $2\ell$, so that it satisfies the automorphy condition \eqref{eqn:automorphy} with $k = \ell$. When $k \geq \ell$, we have shifted holomorphic Hecke cusp forms of weight $2k$ given by $R_{2k-2}R_{2k-4}\cdots R_{2\ell}f$. Similarly, when $k \leq -\ell$ we have the shifted antiholomorphic Hecke cusp form of weight $2k$ given by $L_{2k+2}L_{2k+4}\cdots L_{-2\ell}\overline{f}$. Note that $L_{2\ell} f = R_{-2\ell}\overline{f} = 0$, so that there are no nonzero shifted cusp forms of weight $2k$ with $-\ell < k < \ell$. If $f$ is $L^2$-normalized with respect to the measure $\dmu$ on $\slzh$, then the corresponding $L^2$-normalized shifted holomorphic or antiholomorphic Hecke cusp forms of weight $2k$ are given by
\[f_k \coloneqq \begin{dcases*}
\sqrt{\frac{\gamer{2\ell}}{\gamer{k + \ell}\gamer{k - \ell + 1}}} R_{2k-2}\cdots R_{2\ell} f & for $k \geq \ell$, \\
\sqrt{\frac{\gamer{2\ell}}{\gamer{-k + \ell}\gamer{-k - \ell + 1}}} L_{2k+2}\cdots L_{-2\ell} f & for $k \leq -\ell$.
\end{dcases*}\]
Once more, the $L^2$-normalization of each $f_k$ follows from the fact that $-L_{2k + 2}$ is adjoint to $R_{2k}$ (cf.\ \cite[Corollary 4.4]{DFI02}). The associated lift to $\slzr$ is the function $\Psi_{F,k}(z,\theta) \coloneqq f_k(z) e^{2ki\theta}$, which is an eigenfunction of the Casimir operator $\Omega$ with eigenvalue $\ell (1 - \ell)$.
\item We have the Eisenstein series of weight $2k$, which is defined by
\begin{equation}
\label{eqn:E2kdef}
E_{2k}(z, s) \coloneqq \sum_{\gamma \in \Gamma_\infty \backslash \slz} j_{\gamma}(z)^{-2k} \Im(\gamma z)^s,
\end{equation}
where
\[j_\gamma (z) \coloneqq \frac{cz+d}{|cz+d|}\text{ for } \gamma = \begin{pmatrix}
 a & b \\ c & d
\end{pmatrix}. \]
This series converges absolutely when $\Re(s) > 1$, and can be holomorphically extended to the line $\Re(s) = \frac{1}{2}$. For $s = \frac{1}{2} + it$, $E_{2k}(z, \frac{1}{2} + it)$ is an eigenfunction of $\Delta_{2k}$ with eigenvalue $\frac{1}{4} + t^2$. Letting $E(z,s) = E_0(z,s)$, we note that
\[E_{2k}\left(z, \frac{1}{2} + it\right) = \begin{dcases*}
\frac{\gamer{\frac{1}{2}+it}}{\gamer{\frac{1}{2}+k+it}}R_{2k-2}\cdots R_2R_0 E\left(z, \frac{1}{2} + it\right) & for $k \geq 0$, \\
\frac{\gamer{\frac{1}{2}+it}}{\gamer{\frac{1}{2}-k+it}} L_{2k+2}\cdots L_{-2}L_0 E\left(z, \frac{1}{2} + it\right) & for $k \leq 0$.
\end{dcases*}\]
The associated lift to $\slzr$ is $\tilda{E}_{2k}(z,\theta,\frac{1}{2} + it) \coloneqq E_{2k}(z,\frac{1}{2} + it)e^{2ki\theta}$, which is an eigenfunction of the Casimir operator $\Omega$ with eigenvalue $\frac{1}{4} + t^2$.
\end{itemize}
These Laplacian eigenfunctions satisfy orthonormality relations: we have that
\begin{align*}
\left\langle \Phi_{j,k},1\right\rangle & = \left\langle \Psi_{F,k},1\right\rangle = 0,	\\
\left\langle \Phi_{j,k_1},\tilda{E}_{2k_2}\left(\cdot,\cdot, \frac{1}{2}+it\right)\right\rangle & = \left\langle \Psi_{F,k_1},\tilda{E}_{2k_2}\left(\cdot,\cdot, \frac{1}{2}+it\right)\right\rangle = 0,	\\
\left\langle \Phi_{j,k_1},\Psi_{F,k_2}\right\rangle & = 0,	\\
\left\langle \Phi_{j_1,k_1},\Phi_{j_2,k_2}\right\rangle & = \begin{dcases*}
1 & if $j_1 = j_2$ and $k_1 = k_2$,	\\
0 & otherwise,	\\
\end{dcases*}	\\
\left\langle \Psi_{F_1,k_1},\Psi_{F_2,k_2}\right\rangle & = \begin{dcases*}
1 & if $F_1 = F_2$ and $k_1 = k_2$,	\\
0 & otherwise.
\end{dcases*}
\end{align*}
The Fourier--Whittaker expansions of $\ph_{j,k}$, $f_k$, and $E_{2k}$ are given in Section \ref{sec:fourier}.

\subsection{Spectral Decomposition}

We state the spectral decomposition of $L^2(\slzr)$ below, which follows by combining \cite[Corollary to Theorem 2.3.4]{Bum97} with \cite[Propositions 4.1, 4.2, and 4.3]{DFI02}; for a general reference in the ad\`{e}lic setting, see \cite[Theorem 1.3]{Wu17}. Given $a \in L^2(\slzr)$, we have the spectral decomposition
\begin{multline*}
a(z,\theta)= \frac{3}{\pi} \left\langle a, 1\right\rangle + \sum_{\ell = 1}^{\infty} \sum_{k = -\infty}^{\infty} \left\langle a, \Phi_{\ell,k} \right\rangle \Phi_{\ell,k}(z,\theta) + \sum_{\ell = 1}^{\infty} \sum_{F \in \mathcal{H}_{\ell}} \sum_{\substack{k = -\infty \\ |k| \geq \ell}}^{\infty} \left\langle a, \Psi_{F,k} \right\rangle \Psi_{F,k}(z,\theta) \\
+ \frac{1}{4 \pi} \sum_{k = -\infty}^{\infty} \int_{-\infty}^\infty \left\langle a, \tilda{E}_{2k}\left(\cdot,\cdot, \frac{1}{2}+it\right)\right\rangle \tilda{E}_{2k}\left(z,\theta,\frac{1}{2}+it\right) \, \dt.
\end{multline*}
This converges in the $L^2$-sense. If moreover $a$ is smooth and compactly supported, then this converges absolutely and uniformly on compact sets.

We additionally have Parseval's identity: for $a_1,a_2 \in L^2(\slzr)$, we have the absolutely convergent spectral expansion
\begin{multline}
\label{eqn:Parseval}
\langle a_1,a_2\rangle = \frac{3}{\pi} \left\langle a_1, 1\right\rangle \left\langle 1,a_2\right\rangle + \sum_{\ell = 1}^{\infty} \sum_{k = -\infty}^{\infty} \left\langle a_1, \Phi_{\ell,k} \right\rangle \left\langle \Phi_{\ell,k},a_2 \right\rangle + \sum_{\ell = 1}^{\infty} \sum_{F \in \mathcal{H}_{\ell}} \sum_{\substack{k = -\infty \\ |k| \geq \ell}}^{\infty} \left\langle a_1, \Psi_{F,k} \right\rangle \left\langle \Psi_{F,k},a_2 \right\rangle \\
+ \frac{1}{4 \pi} \sum_{k = -\infty}^{\infty} \int_{-\infty}^\infty \left\langle a_1, \tilda{E}_{2k}\left(\cdot,\cdot, \frac{1}{2}+it\right)\right\rangle \left\langle \tilda{E}_{2k}\left(\cdot,\cdot, \frac{1}{2}+it\right),a_2\right\rangle \, \dt.
\end{multline}

\subsection{QUE on the Modular Surface}
\label{sec:QUEmodularsurface}

There is a significantly simpler formula for the microlocal lift $\omega_j$ of $\ph_j$ to a measure on $\slzr$. We again recall from \cite[(1.18)]{Zel91} that
\[\domega_j(z, \theta) \coloneqq \ph_j(z)\ol{u_j(z, \theta)} \, \domega(z, \theta), \quad u_j(z,\theta) \coloneqq \frac{3}{\pi} \sum_{k = -\infty}^{\infty}\varphi_{j, k}(z)e^{2ki\theta},\]
where convergence of the sum defining $u_j$ is in distribution (i.e.\ $\varphi_j\overline{u_j} \, \domega$ is the limit of measures of the partial sums defining $u_j$)\footnote{More precisely, the measure is defined by $\int a \, \domega_j = \lim_{K \to \infty}\frac{3}{\pi} \int a\varphi_j\sum_{k=-K}^K\ol{\varphi_{j,k}}e^{2ki\theta}\, \domega$.}. In particular, we have that
\begin{align}
\label{eqn:Phiellkdomegaj}
\int_{\slzr} \Phi_{\ell,k}(z,\theta) \, \domega_j(z,\theta) & = \frac{3}{\pi} \int_{\slzh} \varphi_j(z) \overline{\varphi_{j,k}(z)} \varphi_{\ell,k}(z) \, \dmu(z),	\\
\label{eqn:PsiFkdomegaj}
\int_{\slzr} \Psi_{F,k}(z,\theta) \, \domega_j(z,\theta) & = \frac{3}{\pi} \int_{\slzh} \varphi_j(z) \overline{\varphi_{j,k}(z)} f_k(z) \, \dmu(z),	\\
\label{eqn:tildaE2kdomegaj}
\int_{\slzr} \tilda{E}_{2k}\left(z,\theta,\frac{1}{2}+it\right) \, \domega_j(z,\theta) & = \frac{3}{\pi} \int_{\slzh} \varphi_j(z) \overline{\varphi_{j,k}(z)} E_{2k}\left(z,\frac{1}{2} + it\right) \, \dmu(z).
\end{align}

Using the spectral decomposition for $L^2(\slzr)$ and \eqref{eqn:Phiellkdomegaj}, \eqref{eqn:PsiFkdomegaj}, and \eqref{eqn:tildaE2kdomegaj}, for any $a \in C_{c,K}^{\infty}(\slzr)$, we may therefore write
\begin{multline}
\label{eqn:clearnerTheorem}
\int_{\slzr} a(z, \theta) \, \domega_j(z, \theta) = \int_{\slzr} a(z,\theta) \, \domega(z,\theta)	\\
+ \frac{3}{\pi} \sum_{\ell = 1}^{\infty} \sum_{k = -\infty}^{\infty} \left\langle a, \Phi_{\ell,k} \right\rangle \int_{\slzh} \varphi_j(z) \overline{\varphi_{j,k}(z)} \varphi_{\ell,k}(z) \, \dmu(z)	\\
+ \frac{3}{\pi} \sum_{\ell = 1}^{\infty} \sum_{F \in \mathcal{H}_{\ell}} \sum_{\substack{k = -\infty \\ |k| \geq \ell}}^{\infty} \left\langle a, \Psi_{F,k} \right\rangle \int_{\slzh} \varphi_j(z) \overline{\varphi_{j,k}(z)} f_k(z) \, \dmu(z) \\
+ \frac{3}{4 \pi^2} \sum_{k = -\infty}^{\infty} \int_{-\infty}^\infty \left\langle a, \tilda{E}_{2k}\left(\cdot, \cdot,\frac{1}{2}+it\right)\right\rangle \int_{\slzh} \varphi_j(z) \overline{\varphi_{j,k}(z)} E_{2k}\left(z,\frac{1}{2} + it\right) \, \dmu(z) \, \dt.
\end{multline}
To establish Theorem \ref{thm:mainthm}, it therefore suffices to bound each of the three integrals \eqref{eqn:Phiellkdomegaj}, \eqref{eqn:PsiFkdomegaj}, and \eqref{eqn:tildaE2kdomegaj}. The next few sections will be dedicated to resolving each individual case.

\section{Relevant Tools for Computation}

\subsection{Fourier--Whittaker Expansions}
\label{sec:fourier}

We explicitly write out the Fourier--Whittaker expansion for shifted Hecke--Maa{\ss} cusp forms, shifted holomorphic or antiholomorphic Hecke cusp forms, and weighted Eisenstein series. These involve Whittaker functions $W_{\alpha,\beta}(y)$, which are certain special functions on $\R_+$ associated to a pair of parameters $\alpha,\beta \in \C$ that decay exponentially as $y$ tends to infinity in the sense that $\lim_{y \to \infty} y^{-\alpha} e^{y/2} W_{\alpha,\beta}(y) = 1$ (cf.\ \cite[Chapter 16]{WW21} and \cite[Sections 9.22--9.23]{GR15}); they satisfy the second order linear ordinary differential equation
\[W_{\alpha,\beta}''(y) + \left(-\frac{1}{4} + \frac{\alpha}{y} + \frac{\frac{1}{4} - \beta^2}{y^2}\right) W_{\alpha,\beta}(y) = 0.\]
In particular cases, these Whittaker functions are simpler: for $\alpha \in \C$, we have that $W_{\alpha,\alpha - 1/2}(y) = y^{\alpha} e^{-y/2}$, while for $\beta \in \C$, we have that $W_{0,\beta}(y) = \sqrt{\frac{y}{\pi}} K_{\beta}(\frac{y}{2})$, where $K_{\beta}(y)$ denotes the modified Bessel function of the second kind.
\begin{itemize}
\item For Hecke--Maa{\ss} cusp forms of weight $0$, we have by \cite[Theorem 3.11.8]{GH11} the Fourier expansion
\begin{equation}
\label{eqn:phjFW}
\ph_j(z)= \sum_{\substack{n = -\infty \\ n \neq 0}}^{\infty} \sgn(n)^{\kappa_j} \rho_j(1) \frac{\lambda_j(|n|)}{\sqrt{|n|}} W_{0,i r_j}(4 \pi |n|y)e(nx).
\end{equation}
Here $\lambda_j(n)$ is the $n$-th Hecke eigenvalue of $\varphi_j$, $\kappa_j \in \{0,1\}$ is as in \eqref{eqn:paritydefeq}, and the first Fourier coefficient $\rho_j(1) \in \R_+$ satisfies
\begin{equation}
\label{eqn:cusp-form-norm}
\rho_j(1)^2 = \frac{\cosh \pi r_j}{2 L(1,\ad \varphi_j)} = \frac{\pi}{2 \Gamma\left(\frac{1}{2} + ir_j\right) \Gamma\left(\frac{1}{2} - ir_j\right) L(1,\ad \varphi_j)}.
\end{equation}
By the Rankin--Selberg method (cf.\ \cite[Section 19]{DFI02}), this ensures that $\varphi_j$ is $L^2$-normalised. The positive constant $L(1,\ad \varphi_j)$ is the value at $s = 1$ of the adjoint $L$-function $L(s,\ad \varphi_j)$ defined in \eqref{eqn:Lsaddefeq} below. One can use the recurrence relations for Whittaker functions \cite[(9.234)]{GR15} to establish that for shifted Maa{\ss} cusp forms of weight $2k$,
\begin{equation}
\label{eqn:phjkFW}
\ph_{j,k}(z) = \sum_{\substack{n = -\infty \\ n \neq 0}}^{\infty} D_{k, r_j}^{\sgn(n)} \sgn(n)^{\kappa_j} \rho_j(1) \frac{\lambda_j(|n|)}{\sqrt{|n|}} W_{\sgn(n)k,i r_j}(4 \pi |n|y) e(nx),
\end{equation}
where we define the constants
\begin{equation}
\label{eqn:Dkrpm}
D_{k, r}^\pm \coloneqq \frac{(-1)^k \Gamma\left(\frac{1}{2}+ir\right)}{\Gamma\left(\frac{1}{2}\pm k+i r\right)}
\end{equation}
for $r \in \C$ and $k \in \Z$. One sees from \cite[Corollary 4.4]{DFI02} that $\ph_{j,k}$ is also $L^2$-normalized.
\item For shifted holomorphic Hecke cusp forms of positive weight $2k$, we may write the unshifted form as $f = y^\ell F$ for some holomorphic Hecke cusp form $F$ of weight $2\ell$. Once more by \cite[Theorem 3.11.8]{GH11}, this has the Fourier expansion
\[f(z) = \sum_{n=1}^\infty \rho_F(1) \frac{\lambda_F(n)}{\sqrt{n}} (4\pi ny)^\ell e(nz) = \sum_{n=1}^\infty \rho_F(1) \frac{\lambda_F(n)}{\sqrt{n}} W_{\ell,\ell-\frac{1}{2}}(4 \pi ny) e(nx),\]
where again $\lambda_F(n)$ is the $n$-th Hecke eigenvalue of $F$ and the first Fourier coefficient $\rho_F(1) \in \R_+$ satisfies
\begin{equation}
\label{eqn:hol-cusp-form-norm}
\rho_F(1)^2 = \frac{\pi}{2 \gamer{2\ell} L(1,\ad F)},
\end{equation}
which ensures that $f$ is $L^2$-normalized by the Rankin--Selberg method. Once more, the positive constant $L(1,\ad F)$ is the value at $s = 1$ of the adjoint $L$-function $L(s,\ad F)$ defined in \eqref{eqn:Lsaddefeq} below. Applying raising operators, we have that
\[(R_{2k-2} \cdots R_{2\ell+2}R_{2\ell} f)(z) = (-1)^{k-\ell} \sum_{n=1}^\infty \rho_F(1) \frac{\lambda_F(n)}{\sqrt{n}} W_{k,\ell-\frac{1}{2}}(4 \pi ny) e(nx).\]
Finally, we see from \cite[Corollary 4.4]{DFI02} and \cite[(4.60)]{DFI02} that in order to $L^2$-normalize such a form, we have the final Fourier expansion
\begin{equation}
\label{eqn:fkFW}
f_k(z) = \sum_{n=1}^\infty C_{k,\ell} \rho_F(1) \frac{\lambda_F(n)}{\sqrt{n}} W_{k,\ell-\frac{1}{2}}(4 \pi ny)e(nx)
\end{equation}
with
\begin{equation}
\label{eqn:Ckl}
C_{k,\ell} \coloneqq (-1)^{k-\ell} \sqrt{\frac{\gamer{2\ell}}{\gamer{k+\ell}\gamer{k-\ell+1}}}.
\end{equation}
Similarly, for shifted antiholomorphic Hecke cusp forms of negative weight $-2k$, we may write the unshifted Hecke cusp form as $\overline{f} = y^{\ell} \overline{F}$. One has
\[\ol{f(z)} = \sum_{n=1}^\infty \rho_F(1) \frac{\lambda_F(n)}{\sqrt{n}} W_{\ell,\ell - \frac{1}{2}}(4\pi ny) e(-nx),\]
so that
\begin{equation}
\label{eqn:f-kFW}
f_{-k}(z) = \sum_{n=1}^\infty C_{k,\ell} \rho_F(1) \frac{\lambda_F(n)}{\sqrt{n}} W_{k,\ell-\frac{1}{2}}(4 \pi ny)e(-nx).
\end{equation}
\item Finally we recall the Fourier expansion of Eisenstein series. Define
\[\lambda(n,t) \coloneqq \sum_{ab = n}a^{it}b^{-it}.\]
For weight $0$ Eisenstein series, we have from \cite[(1.3)]{Jak94} that
\begin{equation}
\label{eqn:EFW}
E\left(z,\frac{1}{2}+it\right) = y^{\frac{1}{2}+it} + \frac{\xi(1-2it)}{\xi(1+2it)}y^{\frac{1}{2}-it} + \sum_{\substack{n = -\infty \\ n \neq 0}}^{\infty} \frac{1}{\xi(1+2it)} \frac{\lambda(|n|,t)}{\sqrt{|n|}}W_{0,it}(4\pi|n|y)e(nx),
\end{equation}
where $\xi(s) \coloneqq \pi^{-s/2}\Gamma(\frac{s}{2})\zeta(s)$ is the completed Riemann zeta function. For weight $2k$ Eisenstein series, we then have that
\begin{multline*}
E_{2k}\left(z,\frac{1}{2}+it\right) = y^{\frac{1}{2}+it} + \frac{(-1)^k \Gamma\left(\frac{1}{2}+it\right)^2}{\Gamma\left(\frac{1}{2}-k+it\right)\Gamma\left(\frac{1}{2}+k+it\right)} \frac{\xi(1-2it)}{\xi(1+2it)}y^{\frac{1}{2}-it}\\
+ \sum_{\substack{n = -\infty \\ n \neq 0}}^{\infty} \frac{D_{k,t}^{\sgn(n)}}{\xi(1 + 2it)} \frac{\lambda(|n|,t)}{\sqrt{|n|}} W_{\sgn(n)k,it}(4 \pi |n|y) e(nx).
\end{multline*}
\end{itemize}

\subsection{\texorpdfstring{$L$}{L}-Functions}
\label{sec:Lfunctions}

We give a quick overview of all the necessary theory surrounding $L$-functions. A general discussion of the theory of $L$-functions and their bounds can be found in \cite[Chapter 5]{IK04}.

Let $\phi$ be either a Hecke--Maa\ss{} cusp form or a holomorphic Hecke cusp form. Such a Hecke cusp form $\phi$ has an associated $L$-function $L(s,\phi)$. Since the Hecke operators $T_n$ satisfy the multiplicativity relation
\[T_mT_n = \sum_{d \mid (m, n)}T_{\frac{mn}{d^2}},\]
the Hecke eigenvalues $\lambda_{\phi}(n)$ must satisfy the corresponding Hecke relations
\[\lambda_{\phi}(m) \lambda_{\phi}(n) = \sum_{d \mid (m, n)} \lambda_{\phi}\left(\frac{mn}{d^2}\right).\]
We may therefore define for $\Re(s) > 1$ the degree $2$ $L$-function
\[L(s,\phi) \coloneqq \sum_{n = 1}^{\infty} \frac{\lambda_{\phi}(n)}{n^s} = \prod_p \frac{1}{1 - \lambda_{\phi}(p)p^{-s} + p^{-2s}}.\]
This can be analytically continued to a holomorphic function on $\C$. We may write the Euler product as
\[L(s, \phi) = \prod_p \frac{1}{(1-\alpha_{\phi,1}(p) p^{-s})^{-1}(1-\alpha_{\phi,2}(p) p^{-s})},\]
where the \emph{Satake parameters} $\alpha_{\phi,1}(p),\alpha_{\phi,2}(p)$ satisfy
\[\alpha_{\phi,1}(p) + \alpha_{\phi,2}(p) = \lambda_{\phi}(p), \qquad \alpha_{\phi,1}(p) \alpha_{\phi,2}(p) = 1.\]

We also define relevant higher degree $L$-functions: for $m \leq 3$, we define the degree $2^m$ $L$-function
\[L(s,\phi_{1} \otimes \cdots \otimes \phi_m) \coloneqq \prod_p \prod_{(b_j) \in \{1,2\}^m} \frac{1}{1- \alpha_{\phi_1,b_1}(p)\cdots \alpha_{\phi_m,b_m}(p)p^{-s}}.\]We additionally define the degree $3$ and degree $6$ $L$-functions
\begin{align}
\label{eqn:Lsaddefeq}
L(s,\ad \phi) & \coloneqq \frac{L(s,\phi \otimes \phi)}{\zeta(s)}, \\
L(s,\ad \phi_1 \otimes \phi_2) & \coloneqq \frac{L(s,\phi_1 \otimes \phi_1 \otimes \phi_2 )}{L(s,\phi_2)},
\end{align}
Each of these $L$-functions has a meromorphic continuation to $\C$. For later use, we will also recall the identities
\begin{align}
\label{eqn:Dirserieslambda2}
\sum_{n=1}^\infty \frac{\lambda_{\phi}(n)^2}{n^s} & = \frac{\zeta(s)L(s,\ad\phi)}{\zeta(2s)},	\\
\label{eqn:Dirserieslambdalambdat}
\sum_{n=1}^\infty \frac{\lambda_{\phi}(n)\lambda(n,t)}{n^{s}} & = \frac{L(s + it,\phi)L(s - it,\phi)}{\zeta(2s)},
\end{align}
which are both valid for $\Re(s) > 1$.

For any such $L$-function $L(s,\Pi)$ of degree $d$, where $\Pi$ is a placeholder for one of the automorphic objects listed above, we have a corresponding gamma factor of the form
\[L_\infty(s,\Pi)=\prod_{i=1}^d \Gamma_\R(s+\mu_i)\]
for some \emph{Langlands parameters} $\mu_i \in \C$, where $\Gamma_\R(s) \coloneqq \pi^{-s/2}\Gamma(\frac{s}{2})$. The completed $L$-function $\Lambda(s,\Pi) \coloneqq L(s,\Pi) L_\infty(s,\Pi)$ has a meromorphic continuation to $\C$ and satisfies a functional equation of the form $\Lambda(1 - s,\Pi) = \epsilon_{\Pi} \Lambda(s,\widetilde{\Pi})$, where the \emph{epsilon factor} $\epsilon_{\Pi}$ is a complex number of absolute value $1$, while $\Lambda(s,\widetilde{\Pi}) = \overline{\Lambda(\overline{s},\Pi)}$.

\subsection{Bounds for \texorpdfstring{$L$}{L}-Functions}

Various $L$-functions will appear in the integrals computed later in the paper. As such, the study of the sizes of our integrals is connected to the study of the sizes of such $L$-functions. In particular, estimating relevant integrals can be reduced to estimating $L(1,\Pi)$ and $L(\frac{1}{2}+it,\Pi)$ for various values of $t$ and $\Pi$. We discuss the specific relevant bounds.

For $\phi$ a Hecke--Maa\ss{} cusp form with spectral parameter $r$, combining the work of \cite[Main Theorem]{GHL94} and \cite[Corollary 1]{Li10} with \eqref{eqn:cusp-form-norm}, we have that
\begin{equation}
\label{eqn:L1adboundsMaass}
\frac{1}{\log r} \ll L(1,\ad\phi) \ll \exp\left(C (\log r)^{\frac{1}{4}} (\log \log r)^{\frac{1}{2}}\right)
\end{equation}
for some absolute constant $C > 0$. Similarly, for $\phi$ a holomorphic Hecke cusp form of weight $\ell$, we have that
\begin{equation}
\label{eqn:L1adboundshol}
\frac{1}{\log \ell} \ll L(1,\ad\phi) \ll (\log \ell)^3,
\end{equation}
where the lower bound again follows from \cite[Main Theorem]{GHL94}, while the upper bound follows from \cite[Proposition 3.2 (i)]{LW06}. Finally, for $t \in \R$, we have the classical bounds \cite[(8.24), Theorem 8.29]{IK04}
\begin{equation}
\label{eqn:zeta1bound}
\frac{1}{(\log(3 + |t|))^{\frac{2}{3}} (\log \log (9 + |t|))^{\frac{1}{3}}} \ll \left|\zeta(1 + it)\right| \ll \frac{\log(3 + |t|)}{\log \log (9 + |t|)}.
\end{equation}
We shall only make use of the lower bounds in \eqref{eqn:L1adboundsMaass}, \eqref{eqn:L1adboundshol}, and \eqref{eqn:zeta1bound}. In particular, the lower bound in \eqref{eqn:L1adboundsMaass} is precisely the cause of the presence of the term $\log \lambda_j$ on the right-hand side of \eqref{eqn:mainthm}.

To discuss values of an $L$-function $L(s,\Pi)$ on the line $\Re(s) = \frac{1}{2}$, we define the \emph{analytic conductor}
\[C(s,\Pi) \coloneqq \prod_{i=1}^d (1+|s+\mu_i|).\]
The analytic conductor can be thought of as measuring the \emph{complexity} of the $L$-function $L(s,\Pi)$. The \emph{convexity bound} bound for such an $L$-function on the line $\Re(s) = \frac{1}{2}$ is
\[L(s,\Pi) \ll_\e C(s,\Pi)^{\frac{1}{4}+\e}.\]
Such a bound is known for all of the $L$-functions that we study below; it is a consequence of the Phragm\'{e}n--Lindel\"{o}f convexity principle, the functional equations for these $L$-functions, and upper bounds for these $L$-functions at the edge of the critical strip \cite[Theorem 2]{Li10}. A \emph{subconvex} bound is a bound of the form
\[L(s,\Pi) \ll C(s,\Pi)^{\frac{1}{4} - \delta}\]
for some fixed $\delta > 0$; in contrast, for some of the $L$-functions that we study below, such a bound is not yet known. The \emph{generalized Lindel\"{o}f hypothesis} is the conjecture that such a subconvex bound holds with $\delta = \frac{1}{4} - \e$ for any fixed $\e > 0$. The generalized Lindel\"{o}f hypothesis would follow as a consequence from the \emph{generalized Riemann hypothesis}, which is the conjecture that the only zeroes of $L(s,\Pi)$ in the critical strip $0 < \Re(s) < 1$ lie on the critical line $\Re(s) = \frac{1}{2}$.

We make this explicit for various $L$-functions of interest to us by recalling the values of the Langlands parameters $\mu_i$ in these cases. An elementary example is the Riemann zeta function, which is of degree $1$: the Langlands parameter is simply $\mu_1 = 0$, so that the convexity bound is
\begin{equation}
\label{eqn:zetaconvexity}
\zeta\left(\frac{1}{2} + it\right) \ll_{\e} (1 + |t|)^{\frac{1}{4} + \e}.
\end{equation}

Next, from \cite[Sections 5.11 and 5.12]{IK04}, when $\ph$ and $\tilda{\ph}$ are Maa\ss{} cusp forms with spectral parameters $r$ and $\tilda{r}$ and parities $\epsilon$ and $\tilda{\epsilon}$, we have that
\begin{align*}
L_\infty(s, \ph) & = \Gamma_\R\left(s + \frac{1 - \epsilon}{2} + ir\right) \Gamma_\R\left(s + \frac{1 - \epsilon}{2} - ir\right),	\\
L_\infty(s, \ad \ph) & = \Gamma_\R(s+2ir) \Gamma_\R(s) \Gamma_\R(s-2ir)\\
L_\infty(s,\ad \ph \otimes \tilda{\ph}) & = \prod_{\pm } \Gamma_\R\left(s + \frac{1 - \tilda{\epsilon}}{2} + 2ir \pm i\tilda{r}\right) \Gamma_\R\left(s + \frac{1 - \tilda{\epsilon}}{2} \pm i\tilda{r}\right) \Gamma_\R\left(s + \frac{1 - \tilda{\epsilon}}{2} - 2ir \pm i\tilda{r}\right).
\end{align*}
In particular, we have the convexity bounds
\begin{align}
\label{eqn:L1/2phconvexity}
L\left(\frac{1}{2}, \ph\right) & \ll_{\e} r^{\frac{1}{2} + \e},	\\
\label{eqn:L1/2+itadphconvexity}
L\left(\frac{1}{2} + it, \ad \ph\right) & \ll_{\e} ((1 + |t|)(1 + |r + t|)(1 + |r - t|))^{\frac{1}{4} + \e},	\\
\label{eqn:L1/2adphitildaphconvexity}
L\left(\frac{1}{2},\ad \ph \otimes \tilda{\ph}\right) & \ll_{\e} (\tilda{r} (r + \tilda{r})(1 + |r - \tilda{r}|))^{\frac{1}{2} + \e}.
\end{align}

For our applications regarding QUE, we need to assume hypothetical improvements upon \eqref{eqn:L1/2+itadphconvexity} and \eqref{eqn:L1/2adphitildaphconvexity} that imply subconvexity in the $r$-aspect but allow for polynomial growth in the $t$-aspect or $\tilde{r}$-aspect, namely bounds of the form
\begin{align*}
L\left(\frac{1}{2} + it, \ad \ph\right) & \ll r^{\frac{1}{2} - 2\delta} (1 + |t|)^A,	\\
L\left(\frac{1}{2},\ad \ph \otimes \tilda{\ph}\right) & \ll r^{1 - 4\delta} \tilde{r}^{2A}
\end{align*}
for some $\delta > 0$ and $A > 0$ (see Theorems \ref{thm:MaassMaassEis} and \ref{thm:MaassMaassMaass}).

Finally, when $\varphi$ is again a Maa\ss{} cusp form with spectral parameter $r$ and $F$ is a holomorphic Hecke cusp form of weight $2\ell > 0$, we have that
\begin{align*}
L_\infty(s, F) & = \Gamma_\R\left(s+\ell+\frac{1}{2}\right) \Gamma_\R\left(s+\ell-\frac{1}{2}\right) \\
L_\infty(s,\ad \ph \otimes F) & = \prod_{\pm } \Gamma_\R\left(s+2ir +\ell \pm \half\right) \Gamma_\R\left(s+\ell\pm \half\right) \Gamma_\R\left(s-2ir +\ell \pm \half\right).
\end{align*}
In particular, we have the convexity bounds
\begin{align}
\label{eqn:L1/2tildaphconvexity}
L\left(\frac{1}{2} + it,F\right) & \ll_{\e} (\ell + |t|)^{\frac{1}{2} + \e},	\\
\label{eqn:L1/2adphitildaphconvexity2}
L\left(\frac{1}{2},\ad \ph \otimes F\right) & \ll_{\e} (\ell (r + \ell)^2)^{\frac{1}{2} + \e}.
\end{align}
Good \cite[Corollary]{Goo82} has proven an improvement upon \eqref{eqn:L1/2tildaphconvexity} that implies subconvexity in the $t$-aspect, namely the subconvex bound
\begin{equation}
\label{eqn:Goodsubconvex}
L\left(\frac{1}{2}+it,F\right) \ll_{\ell,\e} |t|^{\frac{1}{3}+\e}.
\end{equation}
For our applications regarding QUE, we also need to assume a hypothetical improvement upon \eqref{eqn:L1/2adphitildaphconvexity2} that implies subconvexity in the $r$-aspect, namely a bound of the form
\[L\left(\frac{1}{2},\ad \ph \otimes F\right) \ll_{\ell} r^{1 - 4\delta}\]
for some $\delta > 0$ (see Theorem \ref{thm:MaassMaasshol}).

\section{Completing the Proof of Continuous Spectrum QUE}
\label{sec:fixing}

We now supply the necessary computation missing from Jakobson's proof of QUE for Eisenstein series given in \cite{Jak94}. This setting shares many similarities with that of QUE for Hecke--Maa\ss{} cusp forms; the chief alteration is that the microlocal lift
\[\omega_j(z,\theta) \coloneqq \frac{3}{\pi} \ph_j(z) \sum_{k = -\infty}^{\infty} \ol{\varphi_{j, k}(z)e^{2ki\theta}}\]
of a Hecke--Maa\ss{} cusp form $\varphi_j$ is replaced by the microlocal lift
\[\mu_t(z,\theta) \coloneqq \frac{3}{\pi} E\left(z,\frac{1}{2} + it\right) \sum_{k = -\infty}^{\infty} \ol{E_{2k}\left(z,\frac{1}{2} + it\right) e^{2ki\theta}}\]
of an Eisenstein series $E(z,1/2 + it)$. Similar to the discussion in Section \ref{sec:QUEmodularsurface}, Jakobson's proof of QUE for Eisenstein series requires one to bound both of the integrals
\begin{equation}
\label{eqn:Eisellkdomegaj}
\int_{\slzh} E\left(z,\frac{1}{2}+it\right)E_{-2k}\left(z,\frac{1}{2}-it\right) \varphi_{\ell,k}(z) \, \dmu(z),
\end{equation}
where $\varphi_{\ell,k}$ is a shifted Hecke--Maa{\ss} cusp form of weight $2k \geq 0$ arising from a Hecke--Maa\ss{} cusp form $\varphi_{\ell}$ of weight $0$ and spectral parameter $r_{\ell}$, and
\begin{equation}
\label{eqn:EisFkdomegaj}
\int_{\slzh} E\left(z,\frac{1}{2}+it\right)E_{-2k}\left(z,\frac{1}{2}-it\right) f_k(z) \, \dmu(z),
\end{equation}
where $f_k$ is a shifted holomorphic Hecke cusp form of weight $2k > 0$ obtained by raising a holomorphic Hecke cusp form $F$ of weight $2\ell > 0$ with $\ell < k$. One must similarly also bound an integral involving three Eisenstein series, of which two are shifted; this requires some minor alterations involving incomplete Eisenstein series (see \cite[Section 3]{Jak94}), since otherwise this integral would diverge.

Jakobson treats this altered Eisenstein integral in \cite[Proposition 3.1]{Jak94}, while he treats the shifted Hecke--Maa\ss{} cusp form integral \eqref{eqn:Eisellkdomegaj} in \cite[Proposition 2.2]{Jak94}. For the shifted holomorphic Hecke cusp form integral \eqref{eqn:EisFkdomegaj}, Jakobson only treats the \emph{unshifted} case in \cite[Proposition 2.1]{Jak94}.

To treat the shifted case, we first relate an integral of two Eisenstein series and a shifted holomorphic Hecke cusp form to the product of a ratio of $L$-functions and an integral involving Whittaker functions.

\begin{lemma}
\label{lem:EisEishol}
For any shifted holomorphic Hecke cusp form $f_k$ of weight $2k > 0$ obtained by raising a holomorphic Hecke cusp form $F$ of weight $2\ell > 0$ with $\ell < k$, we have that
\begin{multline}
\label{eqn:EisEishol}
\int_{\slzh}E\left(z,\frac{1}{2}+it\right)E_{-2k}\left(z,\frac{1}{2}-it\right)f_k(z) \, \dmu(z)	\\
= (-1)^{k-\ell} \sqrt{\frac{\pi}{2}} (2\pi)^{1 + 2it} \frac{L\left(\frac{1}{2},F\right) L\left(\frac{1}{2} - 2it,F\right)}{\zeta(1 - 2it)\zeta(1+2it) \sqrt{L(1,\ad F)}} \\
\times \int_0^\infty \frac{W_{0,it}(u)}{\Gamma\left(\frac{1}{2} + it\right)} \frac{W_{k,\ell-\frac{1}{2}}(u)}{\sqrt{\Gamma(k + \ell) \Gamma(k - \ell + 1)}} u^{-\frac{1}{2} - it} \, \frac{\du}{u}.
\end{multline}
\end{lemma}

\begin{proof}
We begin by studying the integral
\[I_1(s) \coloneqq \int_{\slzh}E\left(z,\frac{1}{2} + it\right)E_{-2k}(z,s)f_k(z) \, \dmu(z)\]
when $\Re(s) > 1$, which allows use to make use of the absolutely convergent expression \eqref{eqn:E2kdef} for $E_{-2k}(z,s)$; we later analytically continue $I_1(s)$ to $s = \frac{1}{2} - it$. We first apply the unfolding trick, inserting the identity \eqref{eqn:E2kdef} for $E_{-2k}(z,s)$ and turning the integral over $\slzh$ into one over $\Gamma_\infty \backslash \Hb$ (cf.\ \cite[Proof of Proposition 2.1]{Jak94}). Using the fact that $f_k$ has weight $2k$, we have that
\[I_1(s) = \int_{\Gamma_\infty \backslash \Hb} E\left(z,\frac{1}{2}+it\right) f_k(z) \Im(z)^s \, \dmu(z).\]
We evaluate this integral by taking a fundamental domain of $\Gamma_\infty \backslash \Hb$ to be $[0,1] \times \R_+$. We now insert the Fourier--Whittaker expansions \eqref{eqn:EFW} of $E(z,\frac{1}{2} + it)$ and \eqref{eqn:fkFW} of $f_k(z)$, interchange the order of summation and integration, evaluate the integral over $x \in [0,1]$, and make the substitution $u = 4\pi |n|y$. This leads us to the identity
\[I_1(s) = \frac{(4\pi)^{1-s} C_{k,\ell} \rho_F(1)}{\xi(1+2it)} \sum_{n=1}^\infty \frac{\lambda_F(n)\lambda(n,t)}{n^s} \int_0^\infty W_{0,it}(u)W_{k,\ell-\frac{1}{2}}(u)u^{s-1} \, \frac{\du}{u}.\]
At this point, we analytically continue this expression to $s = \frac{1}{2} - it$, as the Dirichlet series extends holomorphically to the closed half-plane $\Re(s) \geq \frac{1}{2}$ from \eqref{eqn:Dirserieslambdalambdat} (recalling that $\zeta(2s) \neq 0$ for $\Re(s) \geq \frac{1}{2}$), while the integral extends holomorphically to the open half-plane $\Re(s) > \frac{1}{2} - \ell$ by \cite[(7.621.11) and (9.237.3)]{GR15}, since these identities allow us to write the integral as a finite sum of quotients of gamma functions that have no poles for $\Re(s) > \frac{1}{2} - \ell$. Recalling the identities \eqref{eqn:Ckl} for $C_{k,\ell}$, \eqref{eqn:hol-cusp-form-norm} for $\rho_F(1)^2$, and \eqref{eqn:Dirserieslambdalambdat} for the Dirichlet series, we obtain the desired identity.
\end{proof}

\begin{theorem}
\label{thm:EisEishol}
For any shifted holomorphic or antiholomorphic Hecke cusp form $f_k$ of weight $2k$ obtained by raising or lowering a holomorphic Hecke cusp form $F$ of weight $2\ell > 0$ with $\ell < |k|$, we have that
\[\int_{\slzh}E\left(z,\frac{1}{2}+it\right)E_{-2k}\left(z,\frac{1}{2}-it\right)f_k(z) \, \dmu(z) \ll_{k,\ell,\e} |t|^{-\frac{1}{6}+\e}.\]
\end{theorem}

\begin{proof}
We consider only the positive weight case; the analogous bounds for the negative weight case follow by conjugational symmetry. We bound the expression \eqref{eqn:EisEishol}. Via \eqref{eqn:zeta1bound} and the subconvex bound \eqref{eqn:Goodsubconvex}, the ratio of $L$-functions is $O_{\ell,\e}(|t|^{1/3 + \e})$. It remains to deal with the integral of Whittaker functions. In Corollary \ref{cor:Ikellrbound}, we show that this integral is $O_{k,\ell}(|t|^{-1/2})$. This yields the desired estimate.
\end{proof}

\begin{remark}
Theorem \ref{thm:EisEishol} is unconditional due to the fact that the subconvex bound \eqref{eqn:Goodsubconvex} for $L(1/2 + it,F)$ is known unconditionally. A similar such subconvex bound is known for $L(1/2 + it,\varphi)$, where $\varphi$ is a \emph{fixed} Hecke--Maa\ss{} cusp form \cite{Meu90}; finally, an analogous subconvex bound is known for $\zeta(1/2 + it)$. These known subconvex bounds are the key inputs for Jakobson's unconditional proof of QUE for Eisenstein series \cite[Theorem 1]{Jak94}. In contrast, the subconvex bounds \eqref{eqn:subconvex1}, \eqref{eqn:subconvex2}, and \eqref{eqn:subconvex3} remain hypothetical, which is the reason that Theorem \ref{thm:mainthm} is a conditional result.
\end{remark}

\section{Eisenstein Series Computation}

We now move on to the proof of our main theorem, first proving the desired bound for Eisenstein series. We begin by relating an integral of a Hecke--Maa\ss{} cusp form, a shifted Hecke--Maa\ss{} cusp form, and a shifted Eisenstein series to the product of a ratio of $L$-functions and an integral involving Whittaker functions.

\begin{lemma}
\label{lem:EisMaassMaass}
For $k \in \Z$ and $t \in \R$, we have that
\begin{multline}
\label{eqn:EisMaassMaass}
\int_{\slzh}\varphi_j(z)\ol {\varphi_{j, k}(z)}E_{2k}\left(z,\frac{1}{2} + it\right) \, \dmu(z) = \frac{\pi}{2} (-1)^{k + \kappa_j} (4\pi)^{\frac{1}{2} - it} \frac{\zeta\left(\frac{1}{2} + it\right) L\left(\frac{1}{2} + it, \ad \varphi_j\right)}{\zeta(1 + 2it) L(1, \ad\ph_j)}	\\
\times \int_0^\infty \frac{W_{0, ir_j}(u)}{\Gamma\left(\frac{1}{2} + ir_j\right)} \left(\frac{W_{k,-ir_j}(u) }{\Gamma\left(\frac{1}{2}+ k-i r_j\right)} + \frac{W_{-k, -ir_j}(u)}{\Gamma\left(\frac{1}{2}- k-i r_j\right)} \right) u^{-\frac{1}{2} + it} \, \frac{\du}{u}.
\end{multline}
\end{lemma}

\begin{proof}
We follow the same method as in Lemma \ref{lem:EisEishol}, first evaluating the integral
\[I_2(s) \coloneqq \int_{\slzh}\varphi_j(z)\ol {\varphi_{j, k}(z)}E_{2k}(z, s) \, \dmu(z)\]
for $\Re(s) > 1$, and then analytically continuing this expression to $s = \frac{1}{2} + it$. We again apply the unfolding trick by inserting the identity \eqref{eqn:E2kdef} for $E_{2k}(z,s)$, giving
\[I_2(s) =\int_{\Gamma_\infty\backslash\Hb}\varphi_j(z)\ol{\varphi_{j, k}(z)}\Im(z)^s\, \dmu(z).\]
Inserting the Fourier--Whittaker expansions \eqref{eqn:phjFW} for $\ph_j$ and \eqref{eqn:phjkFW} for $\ph_{j, k}$ and integrating over the fundamental domain $[0,1] \times \R_+$ of $\Gamma_{\infty} \backslash \Hb$, we find that $I_2(s)$ is equal to
\[(-1)^{\kappa_j} (4\pi)^{1 - s} \rho_j(1)^2 \sum_{n = 1}^{\infty} \frac{\lambda_j(n)^2}{n^s} \int_0^\infty W_{0, ir_j}(u) \left(\overline{D_{k,-r_j}^+} W_{k,-ir_j}(u) + \overline{D_{k,-r_j}^-} W_{-k,-ir_j}(u)\right) u^{s-1} \, \frac{\du}{u}.\]
We then analytically continue this to $s = \frac{1}{2} + it$, as the Dirichlet series extends meromorphically to the closed half-plane $\Re(s) \geq \frac{1}{2}$ with only a simple pole at $s = 1$ from \eqref{eqn:Dirserieslambda2} (recalling that $\zeta(2s) \neq 0$ for $\Re(s) \geq \frac{1}{2}$), while the integral extends holomorphically to the open half-plane $\Re(s) > 0$ by \cite[(7.611.7)]{GR15}. Recalling the identities \eqref{eqn:Dkrpm} for $D_{k, -r_j}^{\pm}$ (and noting that $\overline{\Gamma(z)} = \Gamma(\overline{z})$), \eqref{eqn:cusp-form-norm} for $\rho_j(1)^2$, and \eqref{eqn:Dirserieslambda2} for the Dirichlet series, we obtain the desired identity.
\end{proof}

\begin{theorem}
\label{thm:MaassMaassEis}
For any $\delta > 0$ and $A > 0$, given a subconvex bound of the form
\begin{equation}
\label{eqn:subconvexEis}
L\left(\frac{1}{2}+it,\ad \phi\right) \ll r^{\frac{1}{2} - 2\delta} (1+|t|)^A,
\end{equation}
where $\phi$ is an arbitrary Hecke--Maa{\ss} cusp form with spectral parameter $r$, we have that
\begin{multline*}
\int_{\slzh}\varphi_j(z)\ol {\varphi_{j, k}(z)}E_{2k}\left(z,\frac{1}{2} + it\right) \, \dmu(z)	\\
\ll_{k,\e} r_j^{\frac{1}{2} - 2\delta} \log r_j \: (1 + |t|)^{A - \frac{1}{4} + \e} (2r_j + |t|)^{-\frac{1}{4}} (1 + |2r_j - |t||)^{-\frac{1}{4}}.
\end{multline*}
\end{theorem}

\begin{proof}
We consider only the positive weight case; the analogous bounds for the negative weight case follow by conjugational symmetry. We bound the expression \eqref{eqn:EisMaassMaass}. Via the assumption of the subconvex bound \eqref{eqn:subconvexEis}, the bounds \eqref{eqn:L1adboundsMaass} and \eqref{eqn:zeta1bound}, and the convexity bound \eqref{eqn:zetaconvexity}, the ratio of $L$-functions in \eqref{eqn:EisMaassMaass} is $O_{\e}(r_j^{1/2 - 2\delta} (\log r_j) (1 + |t|)^{A + 1/4 + \e})$. It remains to deal with the integral of Whittaker functions. In Corollary \ref{cor:Ikellrbound}, we show that this integral is $O_k((1 + |t|)^{-1/2} (2r_j + |t|)^{-1/4} (1 + |2r_j - |t||)^{-1/4})$. This yields the desired estimate.
\end{proof}

\section{The Watson--Ichino Triple Product Formula}

The remaining integrals we wish to compute are of the form
\[\int_{\slzh} \phi_1(z) \phi_2(z) \phi_3(z) \, \dmu(z)\]
where $\phi_i$ are (shifted Maa\ss{}, holomorphic, or antiholomorphic) Hecke cusp forms of weight $2k_i$ for which $k_1 + k_2 + k_3 = 0$. We will compute these via the \emph{Watson--Ichino triple product formula}, which allows us to express these in terms of products of $L$-functions and integrals of Whittaker functions.

The formula given by Ichino \cite[Theorem 1]{Ich08} is extremely general and simplifies greatly when applied to the special case of cusp forms on the modular surface. We follow the simplification of the general formula done in \cite[Appendix]{SZ19}.

Let $\tilda{\phi_i}$ denote the ad\`{e}lic lift of $\phi_i$ to a function on $\mathrm{Z}(\A_\Q) \GL_2(\Q)\backslash \GL_2(\A_\Q)$, as described in \cite[Section 4.3]{HN22} (see also \cite[Section 4.12]{GH11}). We have that
\begin{align}
\notag
\int_{\slzh} \phi_1(z) \phi_2(z) \phi_3(z) \, \dmu(z) & = \int_{\slzr} \phi_1(z) e^{2k_1i\theta} \phi_2(z) e^{2k_2i\theta} \phi_3(z) e^{2k_3i\theta} \, \domega(z,\theta) \\
\label{eqn:intclassicaltoadelic}
& = \frac{\pi}{6} \int_{ \mathrm{Z}(\A_\Q) \GL_2(\Q)\backslash \GL_2(\A_\Q)} \tilda{\phi_1}(g) \tilda{\phi_2}(g) \tilda{\phi_3}(g) \, \dg.
\end{align}
Here $\dg$ denotes the Tamagawa measure on $\mathrm{Z}(\A_\Q) \GL_2(\Q)\backslash \GL_2(\A_\Q)$, which is normalized such that this quotient space has volume $2$. The factor $\frac{\pi}{6}$ occurs on the right-hand side of \eqref{eqn:intclassicaltoadelic} to ensure that the relevant measures are normalized consistently, which can be checked by replacing the integrands with the constant function $1$.

The Watson--Ichino triple product formula relates the integral \eqref{eqn:intclassicaltoadelic} to $L$-functions and to an integral of matrix coefficients of the local representations of $\GL_2(\R)$ associated to $\tilde{\phi_1},\tilde{\phi_2},\tilde{\phi_3}$. So long as one of $\phi_1,\phi_2,\phi_3$ is a shifted Hecke--Maa\ss{} cusp form, this integral of matrix coefficients can in turned be expressed as in terms of an integral of \emph{local Whittaker functions} and an element of the \emph{induced model}, which we describe below. The reduction to an integral of this form is a local analogue of the unfolding method used in Lemmas \ref{lem:EisEishol} and \ref{lem:EisMaassMaass} and leads to integrals of Whittaker functions of the same form as those appearing in \eqref{eqn:EisEishol} and \eqref{eqn:EisMaassMaass}.

\subsection{The Whittaker Model}

Associated to a shifted Maa\ss{}, holomorphic, or antiholomorphic Hecke cusp form $\phi$ of weight $2k$ is a weight $2k$ local Whittaker function $W_{\phi} : \GL_2(\R) \to \C$. This function satisfies
\begin{equation}
\label{eqn:Whittakertransform}
W_{\phi}\left(\begin{pmatrix}
1 & x\\
0 & 1
\end{pmatrix}
\begin{pmatrix}
y & 0\\
0 & 1
\end{pmatrix} \begin{pmatrix} z & 0 \\ 0 & z \end{pmatrix} 
\begin{pmatrix}
\cos\theta & \sin\theta\\
-\sin\theta & \cos\theta
\end{pmatrix}\right) = e(x) e^{2k i\theta} W_{\phi}\begin{pmatrix}
y & 0\\
0 & 1
\end{pmatrix}
\end{equation}
for all $x \in \R$, $y,z \in \R^{\times}$, and $\theta \in \R$; additionally, letting $\lambda_{\phi}(n)$ denote the $n$-th Hecke eigenvalue of $\phi$, we have that for $x \in \R$ and $y \in \R_+$,
\begin{equation}
\label{eqn:adelicliftWhittaker}
\phi(x + iy) = \sum_{\substack{n = -\infty \\ n \neq 0}}^{\infty} \frac{\lambda_{\phi}(|n|)}{\sqrt{|n|}} W_{\phi}\begin{pmatrix} ny & 0 \\ 0 & 1 \end{pmatrix} e(nx)
\end{equation}
(cf.\ \cite[Section 4.3.3]{HN22}). By \eqref{eqn:phjkFW}, \eqref{eqn:fkFW}, and \eqref{eqn:f-kFW}, this means that $W_{\phi}\begin{psmallmatrix} y & 0 \\ 0 & 1 \end{psmallmatrix}$ can be expressed in terms of a constant multiple of a classical Whittaker function $W_{\alpha,\beta}$.

This local Whittaker function $W_{\phi}$ is an element of the \emph{Whittaker model} $\WW(\pi_{\infty})$ associated to $\phi$. As explained in \cite[Section 4.8]{GH11}, associated to $\phi$ is an ad\`{e}lic automorphic form $\widetilde{\phi} : \GL_2(\A_{\Q}) \to \C$. In turn, such an ad\`{e}lic automorphic form is associated to a cuspidal automorphic representation $\pi$ of $\GL_2(\A_{\Q})$, as discussed in \cite[Section 5.4]{GH11}. From \cite[Section 10.4]{GH11}, this automorphic representation is isomorphic to a restricted tensor product of local representations: we have that $\pi \cong \pi_{\infty} \otimes \bigotimes_p' \pi_p$, where $\pi_{\infty}$ is an irreducible representation of $\GL_2(\R)$, while $\pi_p$ is an irreducible representation of $\GL_2(\Q_p)$ for each prime $p$.

The irreducible representation $\pi_{\infty}$ of $\GL_2(\R)$ is completed determined by $\phi$ as follows.
\begin{itemize}
\item If $\phi$ is a shifted Hecke--Maa\ss{} cusp form $\varphi_{j,k}$, then $\pi_{\infty}$ is a principal series representation. A model for this representation is the Whittaker model $\WW(\pi_{\infty})$, which consists of certain local Whittaker functions $W : \GL_2(\R) \to \C$, with the irreducible representation $\pi_{\infty}$ given via the action $(\pi_{\infty}(h) \cdot W)(g) \coloneqq W(gh)$ of $h \in \GL_2(\R)$. Letting $r_j \in \R$ denote the spectral parameter of $\phi$ and $\epsilon_j = (-1)^{\kappa_j}$ denote the parity of $\phi$, where $\kappa_j \in \{0,1\}$, the Whittaker model $\WW(\pi_{\infty})$ is the vector space of local Whittaker functions $W : \GL_2(\R) \to \C$ of the form
\begin{equation}
\label{eqn:WhittakerSchwartzprincipal}
W(g) \coloneqq \sgn(\det g)^{\kappa_j} \left|\det g\right|^{\frac{1}{2} + ir_j} \int_{\R^{\times}} |a|^{-2ir_j} \int_{\R} \Phi((a^{-1},x) g) e(-ax) \, \dx \, \mathrm{d}^{\times}a
\end{equation}
with $\Phi : \R^2 \to \C$ a Schwartz function \cite[Lemma 2.5.13.1]{JL70}.

By taking the Schwartz function to be
\begin{equation}
\label{eqn:Phikprincipal}
\Phi(x_1,x_2) \coloneqq \pi^{|k|} \frac{\Gamma\left(\frac{1}{2} + ir_j\right)}{\Gamma\left(\frac{1}{2} + |k| + ir_j\right)} \rho_j(1) (x_2 - \sgn(k)ix_1)^{2|k|} e^{-\pi(x_1^2 + x_2^2)},
\end{equation}
the resulting local Whittaker function given by \eqref{eqn:WhittakerSchwartzprincipal} is $W_{\phi}$; it satisfies \eqref{eqn:Whittakertransform} and is such that
\begin{equation}
\label{eqn:Wphidiagprincipal}
W_{\phi}\begin{pmatrix} y & 0 \\ 0 & 1 \end{pmatrix} = D_{k,r_j}^{\sgn(y)} \sgn(y)^{\kappa_j} \rho_j(1) W_{\sgn(y)k,ir_j}(4\pi|y|).
\end{equation}
This can be seen directly by taking $g = \begin{psmallmatrix} y & 0 \\ 0 & 1 \end{psmallmatrix}$ in \eqref{eqn:WhittakerSchwartzprincipal} and making the change of variables $x \mapsto a^{-1} xy$ and $a \mapsto \pi^{1/2} |a|^{-1/2} (x^2 + 1)^{1/2} |y|$, which shows that
\begin{multline*}
W_{\phi}\begin{pmatrix} y & 0 \\ 0 & 1 \end{pmatrix} = \pi^{-\frac{1}{2} - ir_j} \frac{(-1)^k \Gamma\left(\frac{1}{2} + ir_j\right)}{\Gamma\left(\frac{1}{2} + |k| + ir_j\right)} \sgn(y)^{\kappa_j} \rho_j(1) |y|^{\frac{1}{2} - ir_j}	\\
\times \int_{0}^{\infty} a^{\frac{1}{2} + |k| + ir_j} e^{-a} \, \frac{\mathrm{d}a}{a} \int_{\R} (1 + ix)^{-\frac{1}{2} + k - ir_j} (1 - ix)^{-\frac{1}{2} - k - ir_j} e(-xy) \, \dx.
\end{multline*}
The integral over $\R_+ \ni a$ is $\Gamma(\frac{1}{2} + |k| + ir_j)$, while from \cite[3.384.9]{GR15},
\[\int_{\R} (1 + ix)^{-\frac{1}{2} + k - ir_j} (1 - ix)^{-\frac{1}{2} - k - ir_j} e(-xy) \, \dx = \frac{\pi^{\frac{1}{2} + ir_j} |y|^{-\frac{1}{2} + ir_j}}{\Gamma\left(\frac{1}{2} + \sgn(y)k + ir_j\right)} W_{\sgn(y)k,ir_j}(4\pi|y|).\]
By the definition \eqref{eqn:Dkrpm} of $D_{k,r}^{\pm}$, this yields \eqref{eqn:Wphidiagprincipal}.
\item If $\phi$ is a shifted holomorphic or antiholomorphic Hecke cusp form $f_k$, then $\pi_{\infty}$ is a discrete series representation. A model for this representation is the Whittaker model $\WW(\pi_{\infty})$, which consists of certain local Whittaker functions $W : \GL_2(\R) \to \C$, with the irreducible representation $\pi_{\infty}$ given via the action $(\pi_{\infty}(h) \cdot W)(g) \coloneqq W(gh)$ of $h \in \GL_2(\R)$. Letting $2\ell \in 2\N$ denote the weight of the underlying holomorphic Hecke cusp form $F$, the Whittaker model $\WW(\pi_{\infty})$ of $\pi_{\infty}$ is the vector space of local Whittaker functions $W : \GL_2(\R) \to \C$ of the form
\begin{equation}
\label{eqn:WhittakerSchwartzdiscrete}
W(g) \coloneqq \left|\det g\right|^{\ell} \int_{\R^{\times}} |y|^{1 - 2\ell} \int_{\R} \Phi((y^{-1},x) g) e(-xy) \, \dx \, \mathrm{d}^{\times}y
\end{equation}
with $\Phi : \R^2 \to \C$ a Schwartz function satisfying
\[\int_{\R} x_1^m \left(\left. \frac{\partial^{2\ell - m - 2}}{\partial x_2^{2\ell - m - 2}} \right|_{x_2 = 0} \int_{\R} \Phi(x_1,\xi_2) e(-x_2 \xi_2) \, \mathrm{d}\xi_2\right) \, \mathrm{d}x_1 = 0 \quad \text{for all $m \in \{0,\ldots,2\ell - 2\}$}\]
\cite[Corollary 2.5.14]{JL70}.

By taking
\[\Phi(x_1,x_2) \coloneqq \pi^{|k|} (-1)^k C_{|k|,\ell} \rho_F(1) (x_2 - \sgn(k)ix_1)^{2|k|} e^{-\pi(x_1^2 + x_2^2)},\]
the resulting local Whittaker function given by \eqref{eqn:WhittakerSchwartzdiscrete} is $W_{\phi}$; it satisfies \eqref{eqn:Whittakertransform} and is such that
\begin{equation}
\label{eqn:Wphidiagdiscrete}
W_{\phi}\begin{pmatrix} y & 0 \\ 0 & 1 \end{pmatrix} = \begin{dcases*}
C_{|k|,\ell} \rho_F(1) W_{|k|,\ell - \frac{1}{2}}(4\pi|y|) & if $\sgn(y) = \sgn(k)$,	\\
0 & if $\sgn(y) = -\sgn(k)$.
\end{dcases*}
\end{equation}
This can again be seen directly by taking $g = \begin{psmallmatrix} y & 0 \\ 0 & 1 \end{psmallmatrix}$ in \eqref{eqn:WhittakerSchwartzdiscrete} and making the change of variables $x \mapsto a^{-1} xy$ and $a \mapsto \pi^{1/2} |a|^{-1/2} (x^2 + 1)^{1/2} |y|$, which shows that
\[W_{\phi}\begin{pmatrix} y & 0 \\ 0 & 1 \end{pmatrix} = \pi^{-\ell} C_{|k|,\ell} \rho_F(1) |y|^{1 - \ell} \int_{0}^{\infty} a^{|k| + \ell} e^{-a} \, \frac{\mathrm{d}a}{a} \int_{\R} (1 + ix)^{-\ell + k} (1 - ix)^{-\ell - k} e(-xy) \, \dx.\]
The integral over $\R_+ \ni a$ is $\Gamma(|k| + \ell)$, while from \cite[3.384.9]{GR15},
\[\int_{\R} (1 + ix)^{-\ell + k} (1 - ix)^{-\ell - k} e(-xy) \, \dx = \begin{dcases*}
\frac{\pi^{\ell} |y|^{\ell - 1}}{\Gamma(|k| + \ell)} W_{|k|,\ell - \frac{1}{2}}(4\pi|y|) & if $\sgn(y) = \sgn(k)$,	\\
0 & if $\sgn(y) = -\sgn(k)$.
\end{dcases*}\]
This yields \eqref{eqn:Wphidiagdiscrete}.
\end{itemize}

\subsection{The Induced Model}

When $\phi$ is a shifted Hecke--Maa\ss{} cusp form $\varphi_{j,k}$, so that the associated irreducible representation $\pi_{\infty}$ of $\GL_2(\R)$ is a principal series representation, there is another natural model for $\pi_{\infty}$ other than the Whittaker model $\WW(\pi_{\infty})$. This is the \emph{induced model} of $\pi_{\infty}$, which consists of smooth functions $f : \GL_2(\R) \to \C$ that satisfy
\[f\left(\begin{pmatrix}
1 & x\\
0 & 1
\end{pmatrix}
\begin{pmatrix}
y & 0\\
0 & 1
\end{pmatrix} \begin{pmatrix} z & 0 \\ 0 & z \end{pmatrix} g\right) = \sgn(y)^{\kappa_j} |y|^{\frac{1}{2} + ir_j} f(g)\]
for all $x \in \R$, $y,z \in \R^{\times}$, and $g \in \GL_2(\R)$. Equivalently, the induced model consists of functions of the form
\begin{equation}
\label{eqn:finduced}
f(g) = \sgn(\det g)^{\kappa_j} \left|\det g\right|^{\frac{1}{2} + ir_j} \int_{\R^{\times}} \Phi((0,a)g) |a|^{1 + 2ir_j} \, \mathrm{d}^{\times}a
\end{equation}
with $\Phi : \R^2 \to \C$ a Schwartz function. There is a bijection between the induced model and the Whittaker model via the map $f \mapsto W$ given by
\[W(g) = \lim_{N \to \infty} \int_{-N}^{N} f\left(\begin{pmatrix} 0 & -1\\ 1 & 0 \end{pmatrix} \begin{pmatrix} 1 & x \\ 0 & 1 \end{pmatrix} g\right) e(-x) \, \dx\]
\cite[Lemma 2.5.13.1]{JL70}.

Taking $\Phi$ as in \eqref{eqn:Phikprincipal}, we see that the element of the induced model $f_{\phi}$ associated to $\phi$ is such that
\begin{equation}
\label{eqn:fphi}
f_{\phi}\left(\begin{pmatrix}
1 & x \\
0 & 1
\end{pmatrix}
\begin{pmatrix}
y & 0 \\
0 & 1
\end{pmatrix} \begin{pmatrix} z & 0 \\ 0 & z \end{pmatrix} 
\begin{pmatrix}
\cos\theta & \sin\theta \\
-\sin\theta & \cos\theta
\end{pmatrix}\right) = \sgn(y)^{\kappa_j} |y|^{\frac{1}{2} + ir_j} e^{2ki \theta} f_{\phi}\begin{pmatrix} 1 & 0 \\ 0 & 1 \end{pmatrix}
\end{equation}
for all $x \in \R$, $y,z \in \R^{\times}$, and $\theta \in \R$. From \eqref{eqn:finduced} together with the change of variables $a \mapsto \pi^{-1/2} |a|^{1/2}$, we have that
\begin{equation}
\label{eqn:fphi2}
f_{\phi}\begin{pmatrix} 1 & 0 \\ 0 & 1 \end{pmatrix} = \pi^{-\frac{1}{2} - ir_j} \Gamma\left(\frac{1}{2} + ir_j\right) \rho_j(1).
\end{equation}

\subsection{The Watson--Ichino Triple Product Formula}

We now state an explicit form of the Watson--Ichino triple product formula, which relates the integral \eqref{eqn:intclassicaltoadelic} to a triple product $L$-function and the square of the absolute value of certain local Whittaker functions and elements of the induced model.

\begin{lemma}[Watson--Ichino triple product formula]
Let $\phi_i$ be Hecke cusp forms of weight $2k_i$ for which $k_1 + k_2 + k_3 = 0$ and such that $\phi_3$ is a shifted Hecke--Maa\ss{} cusp form. Let $W_1$ and $W_2$ denote the local Whittaker functions associated to $\phi_1$ and $\phi_2$ respectively, as in \eqref{eqn:Wphidiagprincipal} and \eqref{eqn:Wphidiagdiscrete}, and let $f_3$ denote the element of the induced model associated to $\phi_3$, as in \eqref{eqn:fphi}. We have that
\begin{multline}
\label{eqn:tripleprod}
\left| \int_{ \mathrm{Z}(\A_\Q) \GL_2(\Q)\backslash \GL_2(\A_\Q)} \tilda{\phi_1}(g) \tilda{\phi_2}(g) \tilda{\phi_3}(g) \, \dg \right|^2	\\
= \frac{36}{\pi^2} L\left(\frac{1}{2},\phi_1 \otimes \phi_2 \otimes \phi_3 \right) \left| \int_{\R^\times} W_1\begin{pmatrix} y & 0 \\ 0 & 1 \end{pmatrix} W_2\begin{pmatrix} y & 0 \\ 0 & 1 \end{pmatrix} f_3\begin{pmatrix} y & 0 \\ 0 & 1 \end{pmatrix} |y|^{-1} \, \mathrm{d}^\times y\right|^2.
\end{multline}
\end{lemma}

\begin{proof}
This follows by combining the Watson--Ichino triple product formula in the form given in \cite[Theorem 1.1]{Ich08} (cf.\ \cite[Theorem 3]{Wat08}) together with the identities \cite[Lemma 5]{SZ19} (cf.\ \cite[Lemma 3.4.2]{MV10}) and \cite[Proposition 6]{Wal85}.
\end{proof}

\begin{remark}
In place of the square of the absolute value of the integral on the right-hand side of \eqref{eqn:tripleprod}, the Watson--Ichino triple product formula given in \cite[Theorem 1.1]{Ich08} instead involves an integral of matrix coefficients. The utility of the induced model is that this integral of matrix coefficients may be expressed in terms of the simpler expression given in \eqref{eqn:tripleprod} \cite[Lemma 5]{SZ19}. In turn, we shall shortly show that for our applications, this simpler expression can be explicitly evaluated in exactly the same way as in the Eisenstein setting in Lemmata \ref{lem:EisEishol} and \ref{lem:EisMaassMaass}. Thus utilizing the induced model allows us to express this integral in a way that is an exact analogue of the Eisenstein integrals in \eqref{eqn:EisEishol} and \eqref{eqn:EisMaassMaass}.
\end{remark}

\section{Maa\texorpdfstring{\ss}{ß} Cusp Form Computation}

We use the Watson--Ichino triple product formula to complete the next step of our main theorem, namely proving the desired bound for Hecke--Maa\ss{} cusp forms. The Watson--Ichino triple product formula allows us to relate an integral of a Hecke--Maa\ss{} cusp form and two shifted Hecke--Maa\ss{} cusp forms to the product of a ratio of $L$-functions and an integral involving Whittaker functions.

\begin{lemma}
For any shifted Hecke--Maa{\ss} cusp form $\varphi_{\ell,k}$ of weight $2k \geq 0$ arising from a Hecke--Maa\ss{} cusp form $\varphi_{\ell}$ of weight $0$ and spectral parameter $r_{\ell}$, we have that
\begin{multline}
\label{eqn:MaassMaassMaass}
\left| \int_{\slzh} \ph_j(z) \ol{\ph_{j,k}(z)} \ph_{\ell,k}(z) \, \dmu(z)\right|^2 = \frac{\pi^3}{2} \frac{L\left(\frac{1}{2},\ph_{\ell}\right) L\left(\frac{1}{2}, \ad \ph_j \otimes \ph_{\ell}\right)}{L(1,\ad \ph_{\ell}) L(1,\ad \ph_j)^2}	\\
\times \left|\int_0^\infty \frac{W_{0, ir_j}(u)}{\Gamma\left(\frac{1}{2} + ir_j\right)} \left(\frac{W_{k,-ir_j}(u) }{\Gamma\left(\frac{1}{2}+ k-i r_j\right)} + \frac{W_{-k, -ir_j}(u)}{\Gamma\left(\frac{1}{2}- k-i r_j\right)} \right) u^{-\frac{1}{2} + ir_{\ell}} \, \frac{\du}{u}\right|^2.
\end{multline} 
\end{lemma}

\begin{proof}
We apply the Watson--Ichino triple product formula \eqref{eqn:tripleprod}, in conjunction with the classical-to-ad\`{e}lic correspondence \eqref{eqn:intclassicaltoadelic}, in the case where the integrand is $\ph_j \ol{\ph_{j,k}}\ph_{\ell,k}$. Thus we set $\phi_1 = \varphi_j$, $\phi_2 = \ol{\ph_{j,k}}$, and $\phi_3 = \ph_{\ell,k}$, and we analyze the right-hand side of \eqref{eqn:tripleprod}. We may factor the triple product $L$-function in \eqref{eqn:tripleprod} as
\[L\left(\frac{1}{2},\ph_{\ell}\right)L\left(\frac{1}{2},\ad \ph_j \otimes \ph_{\ell}\right)\]
via \cite[(9.3) Theorem]{GJ78}. Note that both central $L$-values vanish unless $\ph_{\ell}$ is even, which we assume without loss of generality is the case. We consider the remaining integral in \eqref{eqn:tripleprod}. Recall that $W_1$ and $W_2$ are the local Whittaker functions associated to $\varphi_j$ and $\ol{\ph_{j,k}}$, while $f_3$ is the element of the induced model corresponding to the local Whittaker function $W_3$ for $\ph_{\ell,k}$. From \eqref{eqn:Wphidiagprincipal}, we have that
\begin{equation}
\label{eqn:W1}
W_1\begin{pmatrix} y & 0 \\ 0 & 1 \end{pmatrix} = \sgn(y)^{\kappa_j} \rho_j(1) W_{0,ir_j}(4\pi|y|)
\end{equation}
while comparing \eqref{eqn:phjkFW} and \eqref{eqn:adelicliftWhittaker}, we have that
\[W_2\begin{pmatrix} y & 0 \\ 0 & 1 \end{pmatrix} = D_{k,-r_j}^{\sgn(y)} \sgn(y)^{\kappa_j} \rho_j(1) W_{\sgn(y)k,-ir_j}(4\pi|y|),\]
Finally, we have from \eqref{eqn:fphi} and \eqref{eqn:fphi2} that
\[f_3\begin{pmatrix} y & 0 \\ 0 & 1 \end{pmatrix} = \pi^{-\frac{1}{2} - ir_{\ell}} \Gamma\left(\frac{1}{2} + ir_{\ell}\right) \rho_{\ell}(1) |y|^{\frac{1}{2} + ir_{\ell}},\]
where the assumption that $\varphi_{\ell}$ is even means that we may omit $\sgn(y)^{\kappa_{\ell}}$. Inserting these formul\ae{} and making the substitution $u = 4\pi |y|$, we deduce that
\begin{multline*}
\int_{\R^\times} W_1\begin{pmatrix} y & 0 \\ 0 & 1 \end{pmatrix} W_2\begin{pmatrix} y & 0 \\ 0 & 1 \end{pmatrix} f_3\begin{pmatrix} y & 0 \\ 0 & 1 \end{pmatrix} |y|^{-1} \, \mathrm{d}^\times y = 2 (2\pi)^{-2ir_{\ell}} \Gamma\left(\frac{1}{2} + ir_{\ell}\right) \rho_{\ell}(1) \rho_j(1)^2	\\
\times \int_{0}^{\infty} W_{0,ir_j}(u) \left(D_{k,-r_j}^{+} W_{k,-ir_j}(u) + D_{k,-r_j}^{-} W_{-k,-ir_j}(u)\right) u^{-\frac{1}{2}+ir_{\ell}} \, \frac{\du}{u}.
\end{multline*}
The desired identity now follows from the identities \eqref{eqn:Dkrpm} for $D_{k,-r_j}^{\pm}$ (and noting that $\overline{\Gamma(z)} = \Gamma(\overline{z})$) and \eqref{eqn:cusp-form-norm} for $\rho_{\ell}(1)$ and $\rho_j(1)^2$.
\end{proof}

\begin{theorem}
\label{thm:MaassMaassMaass}
For any $\delta > 0$ and $A > 0$, given a subconvex bound of the form
\begin{equation}
\label{eqn:subconvexMaass}
L\left(\frac{1}{2},\ad \varphi_1 \otimes \varphi_2\right) \ll r_1^{1 - 4\delta} r_2^{2A},
\end{equation}
where $\varphi_1,\varphi_2$ are arbitrary Hecke--Maa{\ss} cusp forms with spectral parameters $r_1,r_2$, we have that
\[\int_{\slzh} \ph_j(z) \ol{\ph_{j,k}(z)} \ph_{\ell,k}(z) \, \dmu(z) \ll_{k,\e} r_j^{\frac{1}{2} - 2\delta} \log r_j \: r_{\ell}^{A - \frac{1}{4} + \e} (2r_j + r_{\ell})^{-\frac{1}{4}} (1 + |2r_j - r_{\ell}|)^{-\frac{1}{4}}\]
for any shifted Hecke--Maa{\ss} cusp form $\ph_{\ell,k}$ of weight $2k$ and spectral parameter $r_{\ell}$.
\end{theorem}

\begin{proof}
We consider only the positive weight case; the analogous bounds for the negative weight case follow by conjugational symmetry. We bound the expression \eqref{eqn:MaassMaassMaass}. Via the assumption of the subconvex bound \eqref{eqn:subconvexMaass}, the bound \eqref{eqn:L1adboundsMaass}, and the convexity bound \eqref{eqn:L1/2phconvexity}, the ratio of $L$-functions in \eqref{eqn:MaassMaassMaass} is $O_{\e}(r_j^{1 - 4\delta} (\log r_j)^2 r_{\ell}^{2A + 1/2 + \e})$. It remains to deal with the integral of Whittaker functions. In Corollary \ref{cor:Ikrrtbound}, we show that this integral is $O_k(r_{\ell}^{-1/2} (2r_j + r_{\ell})^{-1/4} (1 + |2r_j - r_{\ell}|)^{-1/4})$. This yields the desired estimate.
\end{proof}

\section{Holomorphic Cusp Form Computation}

We once more use the Watson--Ichino triple product formula in order to complete the final step of our main theorem, namely proving the desired bound for holomorphic or antiholomorphic Hecke cusp forms. The Watson--Ichino triple product formula allows us to relate an integral of a Hecke--Maa\ss{} cusp form, a shifted Hecke--Maa\ss{} cusp form, and a shifted holomorphic or antiholomorphic Hecke cusp form to the product of a ratio of $L$-functions and an integral involving Whittaker functions.

\begin{lemma}
For any shifted holomorphic Hecke cusp form $f_k$ of weight $2k > 0$ arising from a holomorphic Hecke cusp form $F$ of weight $2\ell > 0$, we have that
\begin{multline}
\label{eqn:MaassMaasshol}
\left| \int_{\slzh} \ph_j(z) \ol{\ph_{j,k}(z)} f_{k}(z) \, \dmu(z)\right|^2 = \frac{\pi^3}{2} \frac{L\left(\frac{1}{2},F\right) L\left(\frac{1}{2}, \ad \ph_j \otimes F\right)}{L(1,\ad F) L(1,\ad \ph_j)^2}	\\
\times \left|\int_0^\infty \frac{W_{0,i r_j}(u)}{\Gamma\left(\frac{1}{2} + ir_j\right)} \frac{W_{k,\ell-\frac{1}{2}}(u)}{\sqrt{\Gamma(k + \ell) \Gamma(k - \ell + 1)}} u^{-\frac{1}{2} - ir_j} \, \frac{\du}{u}\right|^2.
\end{multline} 
\end{lemma}

\begin{proof}
We apply the Watson--Ichino triple product formula \eqref{eqn:tripleprod}, in conjunction with the classical-to-ad\`{e}lic correspondence \eqref{eqn:intclassicaltoadelic}, in the case where the integrand is $\ph_j \ol{\ph_{j,k}}f_k$. Thus we set $\phi_1 = \varphi_j$, $\phi_2 = f_k$, and $\phi_3 = \ol{\ph_{j,k}}$. We may again factor the triple product $L$-function in \eqref{eqn:tripleprod} as
\[L\left(\frac{1}{2},F\right)L\left(\frac{1}{2},\ad \ph_j \otimes F\right)\]
via \cite[(9.3) Theorem]{GJ78}. Both central $L$-values vanish unless $\ell$ is even, which we assume without loss of generality is the case. We consider the remaining integral in \eqref{eqn:tripleprod}. Here $W_1$ is once more as in \eqref{eqn:W1}. Next, $W_2$ is the local Whittaker function associated to $f_k$, so that from \eqref{eqn:Wphidiagdiscrete},
\[W_2\begin{pmatrix} y & 0 \\ 0 & 1 \end{pmatrix} = \begin{dcases*}
C_{k,\ell} \rho_F(1) W_{k,\ell - \frac{1}{2}}(4\pi|y|) & if $y > 0$,	\\
0 & if $y < 0$,
\end{dcases*}\]
noting that $k$ is assumed to be positive. From \eqref{eqn:fphi} and \eqref{eqn:fphi2}, the element of the induced model associated to $\ol{\ph_{j,k}}$ satisfies
\[f_3\begin{pmatrix} y & 0 \\ 0 & 1 \end{pmatrix} = \pi^{-\frac{1}{2} + ir_j} \Gamma\left(\frac{1}{2} - ir_j\right) \sgn(y)^{\kappa_j} \rho_j(1) |y|^{\frac{1}{2} - ir_j}.\]
Inserting these formul\ae{} and making the substitution $u = 4\pi |y|$, we deduce that
\begin{multline*}
\int_{\R^\times} W_1\begin{pmatrix} y & 0 \\ 0 & 1 \end{pmatrix} W_2\begin{pmatrix} y & 0 \\ 0 & 1 \end{pmatrix} f_3\begin{pmatrix} y & 0 \\ 0 & 1 \end{pmatrix} |y|^{-1} \, \mathrm{d}^\times y	\\
= 2 (2\pi)^{2ir_j} \gamer{\frac{1}{2}-ir_j} C_{k,\ell} \rho_F(1) \rho_j(1)^2 \int_{0}^{\infty} W_{0,ir_j}(u) W_{k,\ell-\frac{1}{2}}(u) u^{-\frac{1}{2}-ir_j} \, \frac{\du}{u}.
\end{multline*}
The desired identity now follows from the identities \eqref{eqn:Ckl} for $C_{k,\ell}$, \eqref{eqn:hol-cusp-form-norm} for $\rho_F(1)$, and \eqref{eqn:cusp-form-norm} for $\rho_j(1)^2$.
\end{proof}

\begin{theorem}
\label{thm:MaassMaasshol}
For any $\delta > 0$, given a subconvex bound of the form
\begin{equation}
\label{eqn:subconvexhol}
L\left(\frac{1}{2},\ad \phi \otimes F\right) \ll_{\ell} r^{1 - 4\delta},
\end{equation}
where $\phi$ is an arbitrary Hecke--Maa{\ss} cusp form with spectral parameters $r$ and $F$ is a holomorphic Hecke cusp form of weight $2 \ell > 0$, we have that
\[\int_{\slzh} \ph_j(z) \ol{\ph_{j,k}(z)} f_k(z) \, \dmu(z) \ll_{k,\ell} r_j^{-2\delta} \log r_j\]
for any shifted holomorphic or antiholomorphic Hecke cusp form $f_k$ of weight $2k$ arising from a holomorphic Hecke cusp form $F$ of weight $2\ell > 0$ for which $\ell \leq |k|$.
\end{theorem}

\begin{proof}
We consider only the positive weight case; the analogous bounds for the negative weight case follow by conjugational symmetry. We bound the expression \eqref{eqn:MaassMaasshol}. Via the assumption of the subconvex bound \eqref{eqn:subconvexhol} and the bound \eqref{eqn:L1adboundsMaass}, the ratio of $L$-functions in \eqref{eqn:MaassMaasshol} is $O_{\ell,\e}(r_j^{1 - 4\delta} (\log r_j)^2)$. It remains to deal with the integral of Whittaker functions. In Corollary \ref{cor:Ikellrbound}, we show that this integral is $O_{k,\ell}(r_j^{-1/2})$. This yields the desired estimate.
\end{proof}

\section{Putting Everything Together}

In this section, we prove Theorem \ref{thm:mainthm}.

\begin{proof}[Proof of Theorem \ref{thm:mainthm}]
Let $a \in C_{c,K}^\infty(\slzr)$. We recall from \eqref{eqn:clearnerTheorem} that
\begin{multline*}
\int_{\slzr} a(z, \theta) \, \domega_j(z, \theta) = \int_{\slzr} a(z,\theta) \, \domega(z,\theta)	\\
+ \frac{3}{\pi} \sum_{\ell = 1}^{\infty} \sum_{k = -\infty}^{\infty} \left\langle a, \Phi_{\ell,k} \right\rangle \int_{\slzh} \varphi_j(z) \overline{\varphi_{j,k}(z)} \varphi_{\ell,k}(z) \, \dmu(z)	\\
+ \frac{3}{\pi} \sum_{\ell = 1}^{\infty} \sum_{F \in \mathcal{H}_{\ell}} \sum_{\substack{k = -\infty \\ |k| \geq \ell}}^{\infty} \left\langle a, \Psi_{F,k} \right\rangle \int_{\slzh} \varphi_j(z) \overline{\varphi_{j,k}(z)} f_k(z) \, \dmu(z) \\
+ \frac{3}{4 \pi^2} \sum_{k = -\infty}^{\infty} \int_{-\infty}^\infty \left\langle a, \tilda{E}_{2k}\left(\cdot, \cdot,\frac{1}{2}+it\right)\right\rangle \int_{\slzh} \varphi_j(z) \overline{\varphi_{j,k}(z)} E_{2k}\left(z,\frac{1}{2} + it\right) \, \dmu(z) \, \dt.
\end{multline*}
Since $\Phi_{\ell,k}$, $\Psi_{F,k}$, and $\widetilde{E}_{2k}(\cdot,\cdot,\frac{1}{2} + it)$ are eigenfunctions of the Casimir operator $\Omega$, we have that for any nonnegative integer $A$,
\begin{align*}
\Omega^A \Phi_{\ell,k} & = \left(\frac{1}{4} + r_{\ell}^2\right)^A\Phi_{\ell,k},	\\
\Omega^A \Psi_{F,k} & = (\ell(1 - \ell))^A \Psi_{F,k},	\\
\Omega^A \tilda{E}_{2k}\left(\cdot, \cdot,\frac{1}{2}+it\right) & = \left(\frac{1}{4} + t^2\right)^A \tilda{E}_{2k}\left(\cdot, \cdot,\frac{1}{2}+it\right).
\end{align*}
By the linearity of the inner product together with the fact that the Casimir operator is self-adjoint with respect to the inner product, we deduce that
\begin{align*}
\left\langle a, \Phi_{\ell,k} \right\rangle & = \left(\frac{1}{4} + r_{\ell}^2\right)^{-A} \left\langle \Omega^A a, \Phi_{\ell,k} \right\rangle,	\\
\left\langle a, \Psi_{F,k} \right\rangle & = (\ell(1 - \ell))^{-A} \left\langle \Omega^A a, \Psi_{F,k} \right\rangle,	\\
\left\langle a, \tilda{E}_{2k}\left(\cdot, \cdot,\frac{1}{2}+it\right) \right\rangle & = \left(\frac{1}{4} + t^2\right)^{-A} \left\langle \Omega^A a, \tilda{E}_{2k}\left(\cdot, \cdot,\frac{1}{2}+it\right) \right\rangle.
\end{align*}
Moreover, since $a$ is $K$-finite, there exists a nonnegative integer $M$ for which
\[\left\langle a, \Phi_{\ell,k} \right\rangle = \left\langle a, \Psi_{F,k} \right\rangle = \left\langle a, \tilda{E}_{2k}\left(\cdot, \cdot,\frac{1}{2}+it\right)\right\rangle = 0\]
whenever $|k| > M$. From Theorems \ref{thm:MaassMaassEis}, \ref{thm:MaassMaassMaass}, and \ref{thm:MaassMaasshol}, we deduce that
\begin{multline*}
\int_{\slzr} a(z, \theta) \, \domega_j(z, \theta) - \int_{\slzr} a(z,\theta) \, \domega(z,\theta)	\\
\ll_{M,\e} r_j^{\frac{1}{2} - 2\delta} \log r_j \sum_{\ell = 1}^{\infty} \sum_{k = -M}^{M} \left|\left\langle \Omega^{\left\lceil \frac{A + 1}{2}\right\rceil} a, \Phi_{\ell,k} \right\rangle\right| r_{\ell}^{-\frac{5}{4} + \e} (2r_j + r_{\ell})^{-\frac{1}{4}} (1 + |2r_j - r_{\ell}|)^{-\frac{1}{4}}	\\
+ r_j^{-2\delta} \log r_j \sum_{\ell = 1}^{M} \sum_{F \in \mathcal{H}_{\ell}} \sum_{\substack{k = -M \\ |k| \geq \ell}}^{M} \left|\left\langle \Omega^{\left\lceil \frac{A + 1}{2}\right\rceil} a, \Psi_{F,k} \right\rangle\right| \ell^{-A - 1} \\
+ r_j^{\frac{1}{2} -2\delta} \log r_j \sum_{k = -M}^{M} \int_{-\infty}^\infty \left|\left\langle \Omega^{\left\lceil \frac{A + 1}{2}\right\rceil} a, \tilda{E}_{2k}\left(\cdot, \cdot,\frac{1}{2}+it\right)\right\rangle\right| (1 + |t|)^{-\frac{5}{4} + \e} (2r_j + |t|)^{-\frac{1}{4}} (1 + |2r_j - |t||)^{-\frac{1}{4}} \, \dt.
\end{multline*}
Weyl's law \cite[Theorem 2]{Ris04} implies that $\#\{\ell \in \N : T - 1 < r_{\ell} \leq T\} \ll T$ for $T \geq 1$. Thus by subdividing the sum over $\ell \in \N$ into sums for which $r_{\ell} \in (T - 1,T]$ for each $T \in \N$, we find that
\[\sum_{\ell = 1}^{\infty} r_{\ell}^{-\frac{5}{2} + \e} (2r_j + r_{\ell})^{-\frac{1}{2}} (1 + |2r_j - r_{\ell}|)^{-\frac{1}{2}} \ll \frac{1}{r_j}.\]
Similarly,
\[\int_{-\infty}^{\infty} (1 + |t|)^{-\frac{5}{2} + \e} (2r_j + |t|)^{-\frac{1}{2}} (1 + |2r_j - |t||)^{-\frac{1}{2}} \, \dt \ll \frac{1}{r_j}.\]
Thus by the Cauchy--Schwarz inequality and Bessel's inequality (bearing in mind Parseval's identity \eqref{eqn:Parseval}), we deduce that
\begin{align}
\label{eqn:Sobolev}
\int_{\slzr} a(z, \theta) \, \domega_j(z, \theta) - \int_{\slzr} & a(z,\theta) \, \domega(z,\theta)	\\
\notag
& \ll_{M} \left\|\Omega^{\left\lceil \frac{A + 1}{2}\right\rceil} a\right\|_{L^2(\slzr)} r_j^{-2\delta} \log r_j.
\qedhere
\end{align}
\end{proof}

\begin{remark}
\label{rem:evenweight}
Theorem \ref{thm:mainthm} is proven for functions $a : \slzr \to \C$ that are finite linear combinations of even weight smooth compactly supported functions. In order to remove the condition that $a$ be a finite linear combination of even weight functions, we would require bounds for the integral \eqref{eqn:Phiellkdomegaj} that are uniform not only in $r_j$ and $r_{\ell}$ but additionally uniform in $k$; we would also similarly require such uniform bounds for the integrals \eqref{eqn:PsiFkdomegaj} and \eqref{eqn:tildaE2kdomegaj}. To prove such uniform bounds would require stronger bounds for certain hypergeometric functions than the weaker bounds we derive in Lemma \ref{lem:4F3bound} and Corollary \ref{cor:Ikellrbound} below, as we discuss in Remark \ref{rem:uniform}.
\end{remark}

\appendix

\section{Whittaker Integral Computations}

\subsection{Special Functions}

We compute integrals of Whittaker functions by expressing them in terms of generalized hypergeometric functions, as defined in \cite[Chapter 2]{Sla66}. A generalized hypergeometric function is defined, wherever it converges, as a series
\begin{equation}
\label{eqn:hypergeometric}
\pFqx{p}{q}{a_1,\dots,a_p}{b_1,\dots,b_q}{z} \coloneqq \sum_{m=0}^\infty\frac{(a_1)_m\cdots(a_p)_m}{(b_1)_m\cdots(b_q)_m} \frac{z^m}{m!}.
\end{equation}
Here $(b)_m \coloneqq b (b+1) \cdots (b+m-1)$ and $(b)_0 \coloneqq 1$ for all $b \in \C$, so that
\begin{equation}
\label{eqn:Pochhammer}
(b)_m = \begin{dcases*}
\frac{\Gamma(b + m)}{\Gamma(b)} & if $b$ is not a nonpositive integer,	\\
(-1)^m \frac{\Gamma(1 - b)}{\Gamma(1 - b - m)} & if $b$ is a nonpositive integer and $m \leq -b$,	\\
0 & if $b$ is a nonpositive integer and $m > -b$.
\end{dcases*}
\end{equation}
To bound hypergeometric functions, we must therefore bound gamma functions. We do this via Stirling's formula, which states that for $s \in \C$ with $\Re(s) > \delta$ with $\delta > 0$,
\[\Gamma(s) = \sqrt{2\pi} s^{s - \frac{1}{2}} e^{-s} \left(1 + O_{\delta}\left(\frac{1}{|s|}\right)\right).\]
We use this in the following form: for $s = \sigma + i\tau$ with $\sigma > 0$,
\begin{equation}
\label{eqn:Stirling}
|\Gamma(\sigma + i\tau)| \asymp_{\sigma} (1 + |\tau|)^{\sigma - \frac{1}{2}} e^{-\frac{\pi}{2}|\tau|}.
\end{equation}

\subsection{Non-Holomorphic Case}

We seek to provide an upper bound for an integral of the form
\[I_k(\alpha,\beta,\gamma)=\int_0^\infty \frac{W_{0,i\alpha}(y)}{\gamer{\half+i\alpha}} \left(\frac{W_{k,i\beta}(y)}{\gamer{\half+k+i\beta}}+\frac{W_{-k,i\beta}(y)}{\gamer{\half-k+i\beta}} \right) y^{-\half+i\gamma} \, \frac{\dy}{y},\]
where $k \in \Z$ and $\alpha,\beta,\gamma \in \R$. This can be expressed in terms of gamma functions and a terminating hypergeometric function.

\begin{lemma}[{\cite[(27)]{Jak97}}]
For $k \in \Z$ and $\alpha,\beta,\gamma \in \R$, we have that
\begin{multline}
\label{eqn:hypergeom}
I_k(\alpha,\beta,\gamma) = \frac{(-1)^k 4^{i\gamma}}{2\pi} \frac{\prod_{\epsilon_1, \epsilon_2 \in \{\pm 1\}} \gamer{\frac{1}{4}+\frac{i}{2}\left(\epsilon_1 \alpha + \epsilon_2 \beta + \gamma \right)}}{\gamer{\half+i\alpha}\gamer{\half+i\beta}\gamer{\half+i\gamma}}\\
\times \pFq{4}{3}{-k,k,\frac{1}{4}+\frac{i}{2}(-\alpha+\beta+\gamma),\frac{1}{4}+\frac{i}{2}(\alpha+\beta+\gamma) }{\half, \half+i\beta, \half+i\gamma}. 
\end{multline}
\end{lemma}

To obtain uniform bounds for the expression \eqref{eqn:hypergeom}, we first deal with the ratio of gamma functions.

\begin{lemma}
\label{lem:gammaStirling}
For $r,t \in \R$, we have that
\begin{multline*}
\frac{\gamer{\frac{1}{4}+\frac{i(2r + t)}{2}} \gamer{\frac{1}{4}+\frac{it}{2}}^2 \gamer{\frac{1}{4}+\frac{i(-2r + t)}{2}}}{\gamer{\half+ir}\gamer{\half-ir}\gamer{\half+it}}	\\
\ll \begin{dcases*}
(1 + |t|)^{-\frac{1}{2}} (1 + |2r + t|)^{-\frac{1}{4}} (1 + |2r - t|)^{-\frac{1}{4}} & if $|t| \leq 2|r|$,	\\
(1 + |t|)^{-\frac{1}{2}} (1 + |2r + t|)^{-\frac{1}{4}} (1 + |2r - t|)^{-\frac{1}{4}} e^{-\frac{\pi}{2}(|t| - 2|r|)} & if $|t| \geq 2|r|$.
\end{dcases*}
\end{multline*}
\end{lemma}

\begin{proof}
This follows from Stirling's formula \eqref{eqn:Stirling}.
\end{proof}

Next, we bound the hypergeometric function in \eqref{eqn:hypergeom}.

\begin{lemma}
\label{lem:4F3bound}
For $k \in \Z$ and $r,t \in \R$, we have that
\[\pFq{4}{3}{-k,k,\frac{1}{4}+\frac{i(-2r + t)}{2},\frac{1}{4}+\frac{it}{2}}{\half, \half-ir, \half+it} \ll_k 1 + \left(\frac{1 + |2r - t|}{1 + |r|}\right)^{|k|}.\]
\end{lemma}

\begin{proof}
By \eqref{eqn:hypergeometric} and \eqref{eqn:Pochhammer}, the left-hand side is
\[\sum_{m = 0}^{|k|} \frac{\sqrt{\pi} |k| (-1)^m (|k| + m - 1)!}{(|k| - m)! \Gamma\left(\frac{1}{2} + m\right) m!} \frac{\Gamma\left(m + \frac{1}{4} + \frac{i(-2r + t)}{2}\right) \Gamma\left(m + \frac{1}{4} + \frac{it}{2}\right) \Gamma\left(\frac{1}{2} - ir\right) \Gamma\left(\frac{1}{2} + it\right)}{\Gamma\left(\frac{1}{4} + \frac{i(-2r + t)}{2}\right) \Gamma\left(\frac{1}{4} + \frac{it}{2}\right) \Gamma\left(\frac{1}{2} + m - ir\right) \Gamma\left(\frac{1}{2} + m + it\right)}.\]
By Stirling's formula \eqref{eqn:Stirling}, each summand is
\[\ll_k \left(\frac{1 + |2r - t|}{1 + |r|}\right)^m.\]
This yields the desired bounds.
\end{proof}

Combining Lemmata \ref{lem:gammaStirling} and \ref{lem:4F3bound}, we deduce the following bounds; these bounds are not sharp but are more than sufficient for our purposes.

\begin{corollary}
\label{cor:Ikrrtbound}
For $k \in \Z$ and $r,t \in \R$,
\[I_k(r,-r,t) \ll_k (1 + |t|)^{-\frac{1}{2}} (1 + |2r + t|)^{-\frac{1}{4}} (1 + |2r - t|)^{-\frac{1}{4}}.\]
\end{corollary}

\subsection{Holomorphic Case}

Here we instead seek to provide an upper bound for an integral of the form
\[I_{k,\ell}(r) \coloneqq \int_0^\infty \frac{W_{0,ir}(u)}{\Gamma\left(\frac{1}{2} + ir\right)} \frac{W_{k,\ell-\frac{1}{2}}(u)}{\sqrt{\Gamma(k + \ell) \Gamma(k - \ell + 1)}} u^{-\frac{1}{2} - ir} \, \frac{\du}{u},\]
where $k,\ell \in \N$ are positive integers for which $k \geq \ell$ and $r \in \R$.

\begin{lemma}
For $k,\ell \in \N$ for which $k \geq \ell$ and for $r \in \R$, we have that
\begin{multline}
\label{eqn:Ikellt}
I_{k,\ell}(r) = (-1)^{k - \ell} \sqrt{\frac{\pi}{2}} \frac{\sqrt{\Gamma(k + \ell) \Gamma(k - \ell + 1)}}{\Gamma\left(\frac{1}{2} + ir\right)}	\\
\times \sum_{m = 0}^{k - \ell} \frac{(-1)^m (\ell + m - 1)! \Gamma(\ell + m - 2ir)}{(k - \ell - m)! (2\ell + m - 1)! \Gamma\left(\frac{1}{2} + \ell + m - ir\right) m!}.
\end{multline}
\end{lemma}

\begin{proof}
We use the fact that
\begin{align*}
W_{k,\ell-\frac{1}{2}}(u) & = (-1)^{k - \ell} (k - \ell)! e^{-\frac{u}{2}} u^{\ell} L_{k - \ell}^{(2\ell - 1)}(u)	\\
& = (-1)^{k - \ell} (k - \ell)! (k + \ell - 1)! \sum_{m = 0}^{k - \ell} \frac{(-1)^m}{(k - \ell - m)! (2\ell + m - 1)! m!} u^{\ell + m} e^{-\frac{u}{2}}
\end{align*}
from \cite[(8.970.1) and (9.237.3)]{GR15}, where $L_n^{(\alpha)}$ denotes the associated Laguerre polynomial, together with the identity
\[\int_0^\infty W_{0,ir}(u) e^{-\frac{u}{2}} u^{\ell + m - \frac{1}{2} - ir} \, \frac{\du}{u} = \frac{(\ell + m - 1)! \Gamma(\ell + m - 2ir)}{\Gamma\left(\frac{1}{2} + \ell + m - ir\right)}\]
from \cite[(7.621.11)]{GR15}, in order to obtain the desired identity.
\end{proof}

\begin{remark}
Via \eqref{eqn:hypergeometric} and \eqref{eqn:Pochhammer}, we may write $I_{k,\ell}(r)$ in the form
\[(-1)^{k - \ell} \sqrt{\frac{\pi}{2}} \frac{\sqrt{\Gamma(k + \ell)} \Gamma(\ell) \Gamma(\ell - 2ir)}{\sqrt{\Gamma(k - \ell + 1)} \Gamma(2\ell) \Gamma\left(\frac{1}{2} + ir\right) \Gamma\left(\frac{1}{2} + \ell - ir\right)} \pFq{3}{2}{\ell - k,\ell,\ell - 2ir}{2\ell,\frac{1}{2} + \ell - ir}.\]
One can show that this can alternatively be written as
\begin{multline*}
(-1)^{k - \ell} \sqrt{\frac{\pi}{2}} \frac{\Gamma(\ell)}{\sqrt{\Gamma(k + \ell) \Gamma(k - \ell + 1)}} \frac{\Gamma\left(\frac{1}{2} + k + ir\right) \Gamma(\ell - 2ir)}{\Gamma\left(\frac{1}{2} + ir\right)\Gamma\left(\frac{1}{2} + \ell + ir\right) \Gamma\left(\frac{1}{2} + \ell - ir\right)}	\\
\times \pFq{3}{2}{\ell - k,\frac{1}{2} + ir,\frac{1}{2} - ir}{\frac{1}{2} + \ell + ir,\frac{1}{2} + \ell - ir}.
\end{multline*}
However, we do not make use of these identities.
\end{remark}

We now bound $I_{k,\ell}(r)$.

\begin{corollary}
\label{cor:Ikellrbound}
For $k,\ell \in \N$ for which $k \geq \ell$ and $r \in \R$, we have that
\[I_{k,\ell}(r) \ll_{k,\ell} (1 + |r|)^{-\frac{1}{2}}.\]
\end{corollary}

\begin{proof}
We simply bound each summand in \eqref{eqn:Ikellt} via Stirling's formula \eqref{eqn:Stirling}.
\end{proof}

\begin{remark}
\label{rem:uniform}
As highlighted in Remark \ref{rem:evenweight}, it would be desirable to prove bounds for $I_k(r,-r,t)$ that are uniform not only with respect to $r$ and $t$ but also with respect to $k$. Similarly, it would be desirable to prove bounds for $I_{k,\ell}(r)$ that are uniform not only with respect to $r$ but also with respect to $k$ and $\ell$. A closer examination of the method of proofs of Corollaries \ref{cor:Ikrrtbound} and \ref{cor:Ikellrbound} shows that these methods can be used to give bounds that grow exponentially with $k$, which is insufficient for our needs. Were we to fix every variable except $k$, then the methods in \cite{Fie65} can be used to obtain polynomial bounds for both of $I_k(r,-r,t)$ and $I_{k,\ell}(r)$ solely in the $k$-aspect; unfortunately, however, this is also insufficient for our needs.
\end{remark}

\end{document}